\documentclass[a4paper,11pt]{article}
\usepackage{lmodern}				
\usepackage[utf8]{inputenc}	
\usepackage{amsfonts}
\usepackage{xcolor}
\usepackage[pdftex]{graphicx}
\usepackage{epsfig}
\usepackage{graphicx}
\usepackage{amsmath,amsthm,amssymb}
\usepackage{vmargin}
\setmarginsrb{2cm}{2cm}{2cm}{2cm}{2cm}{2cm}{2cm}{2cm}
\usepackage{textcomp}
\usepackage{url}

\usepackage{algorithm}
\usepackage{ulem}
\usepackage{algorithmic}
\usepackage{dsfont}
\usepackage{tikz}
\usetikzlibrary{arrows}
\usepackage{verbatim}
\usepackage{appendix}

\newcommand{\R}{\ensuremath{\mathbb{R}} }

\newcommand{\QQ}{\ensuremath{\mathbb{Q}} }

\newcommand{\Mb}{\mathbf{M}}
\newcommand{\Mbm}{\mathbf{M}^{-1}}
\newcommand{\E}{\mathbb{E}}
\def \Var{\hbox{{\textrm{Var}}}}

\def \Cov{\hbox{{\textrm{Cov}}}}

\newcommand{\Xn}{X_n}
\newcommand{\Xpn}{X'_{n}}
\newcommand{\Xpnp}{X'_{n+1}}
\newcommand{\Xnp}{X_{n+1}}
\newcommand{\Un}{U_n}
\newcommand{\Unp}{U_{n+1}}
\newcommand{\Yn}{Y_n}
\newcommand{\Ynp}{Y_{n+1}}

\newcommand{\YUn}{Y^{(U)}_{n}}
\newcommand{\YUnp}{Y^{(U)}_{n+1}}
\newcommand{\mn}{\widehat{m}_n}
\newcommand{\mnp}{\widehat{m}_{n+1}}
\newcommand{\Sbn}{\hat{\mathbf{S}}_n}
\newcommand{\Sbk}{\hat{\mathbf{S}}_k}
\newcommand{\Sbun}{\hat{\mathbf{S}}_1}
\newcommand{\Sbncorr}{{\hat{\mathbf{S}}_n}}
\newcommand{\Sbnp}{\hat{\mathbf{S}}_{n+1}}

\newcommand{\Hhnp}{\widehat{\nabla \phi}_{n+1}}
\newcommand{\Sb}{\mathbf{S}^*}
\newcommand{\Sbf}{\mathbf{S}}

\newcommand{\Sc}{\Sigma^*}
\newcommand{\Yu}{Y^{(u)}}

\newcommand{\Ups}{\Upsilon}

\newcommand{\N}{\mathbb{N}}

\usepackage{dsfont}
\newcommand{\un}{\mathds{1}}

\def\Dp{\mathcal{D}_{h}}

\newtheorem{theo}{Theorem}[section]
\newtheorem*{theo*}{Theorem}

\newtheorem{cor}[theo]{Corollary}
\newtheorem{rem}[theo]{Remark}
\newtheorem{prop}[theo]{Proposition}
\newtheorem{lemma}[theo]{Lemma}

\definecolor{no}{RGB}{255,127,14}

\usepackage{authblk}
\begin{document}

\date{\today}
\title{On-line Pick-Freeze Mirror algorithm for Sensitity Analysis}

\author[1]{Manon Costa}
\author[2]{Sébastien Gadat} 
\author[3,4]{Xavier Gendre}
\author[3,5]{Thierry Klein}
\affil[1]{Institut de Math\'ematiques de Toulouse; UMR5219. Universit\'e de Toulouse.}
\affil[2]{Toulouse School of Economics, CNRS UMR 5314, Université Toulouse 1 Capitole, Toulouse. Institut Universitaire de France. }
\affil[3]{Institut de Math\'ematiques de Toulouse; UMR5219. Universit\'e de Toulouse}
\affil[4]{Pathway, Paris, France}
\affil[5]{ENAC - Ecole Nationale de l'Aviation Civile , Universit\'e de Toulouse, France}
%

\maketitle

\begin{abstract}	The main objective of this paper is to propose a new approach for estimating the entire collection of Sobol' indices simultaneously. 
    	Our approach exploits the fact that Sobol' indices can be rewritten as solutions to an optimisation problem over the simplex of $\R^d$, to construct an online sequence of estimators using a stochastic mirror descent algorithm. We prove that our estimation procedure is consistent and provide a non-asymptotic upper bound for its rate of convergence.  Furthermore, we demonstrate the numerical accuracy of our method and compare it with other classical estimation procedures.
\end{abstract}

%
\section{Introduction} \label{sec:intro}

\subsection{Motivation}
The study of how a numerical code's output depends on its input variables has become critically important in many fields, including physics, engineering, applied mathematics, and signal and image processing, among others.
Sensitivity analysis is particularly important in fields where meta-models are used, such as modeling coastal flooding \cite{betancourt:hal-01998724}, optimizing aircraft geometry in aeronautics \cite{peteilh:hal-02866381}, and in various other areas of engineering. Further examples can be found in \cite{broto2020variance,marrel2012global,GJKL14}.
The importance of sensitivity analysis is growing with the advent of artificial intelligence techniques, as the use of deep neural networks with highly sophisticated architectures is gradually replacing costly, traditional, physics-based codes for computing outputs through a regression paradigm, for example \cite{doury2023regional,doury2024suitability}. However, the complexity of these AI-generated models is often perceived as mysterious and difficult to interpret, particularly with regard to how the output is influenced by the input variables. It is nevertheless crucial to understand the effects of input variables on a code’s output, in order to improve the generation of certain input variables and to enhance public acceptance of how the code operates.
Therefore, it is important to quantify the influence of variables on the system with a reasonably limited number of evaluations and computations.

When these inputs are regarded as random elements, this problem is generally referred to as global sensitivity analysis.
Global sensitivity analysis considers the input vector as random and provides a measure of the influence, in terms of output fluctuations, of each subset of its components. We refer to the seminal book \cite{saltelli-sensitivity} for an overview, to \cite{da2021basics} for a synthesis of recent trends in this field, and to \cite{rocquigny2008uncertainty,saltelli-sensitivity,santner2003design,sobol1993} for a discussion of the practical aspects of global sensitivity analysis.

Of the various measures used in global sensitivity analysis, variance-based measures, derived from Hoeffding's decomposition of variance, are probably the most commonly used. When considering an output $Y$ of a computer code modeled as $Y=f(X^1,\ldots,X^p)$, one obtains one of the most common measures of the sensitivity of $Y$ with respect to a subset of inputs ${X^i, i\in u}$ for $u\subset{1,\cdots,p}$: the closed Sobol' index $\Sigma_u$, defined by
 \begin{equation*}
 \Sigma_u=\frac{\Var(\E[Y|X^i, i\in u])}{\Var(Y)}=\frac{\E[\E[Y|X^i, i\in u]^2]-\E[Y]^2}{\Var(Y)}.
 \end{equation*}
In recent years, a myriad of different estimators have been proposed (see \cite[Chapter 4]{da2021basics} for a full review). A first class gathers spectral methods that aim at constructing estimators based on the spectral analysis of the input/output functional relationship and Parseval's formula. We refer to \cite{sudret2008global} for a basic description of the method and to \cite[Chapter 4]{da2021basics} for more recent references. It should be noted that the asymptotic properties of these methods have been little studied, as they are based on the theory of non-linear (quadratic) functional estimation. 
\medskip
A second class of methods requires that the covariates $(X^1,\dots,X^p)$ be independent in order to obtain the asymptotic properties of the estimators.
Among them, the so-called Pick–Freeze design of experiments aims at computing a Monte Carlo estimate of $\Sigma_u$ from a sequence of evaluations of $f$ on a sample of input values, fixing those of the subset $u$. The main advantage is that only minimal assumptions on $f$ are required to derive consistency and central limit theorems. In particular, assumptions of integrability, but not regularity, on $f$ are necessary (see, for example, \cite{pickfreeze, janon2012asymptotic}). Moreover, to estimate at rate $\sqrt{n}$ a single Sobol' index, one needs a design of experiments of size $2n$. This implies that estimating all the $p$ first-order indices involves a sample of size $(p+1)n$, and a $2^p n$ sample is required to estimate all Sobol' indices.
This method has two drawbacks: the size of the DOE required to estimate all indices grows exponentially with the number of input variables, and it relies on the very specific Pick–Freeze experimental design.
\medskip

In order to tackle these limitations, methods based on local averaging have been developed. Among them are kernel estimators, which have been thoroughly studied for the case $d=1$ \cite{da2009local,da2008efficient,plischke2020fighting,solis2021non,heredia2021nonparametric}, and provide central limit theorems and asymptotic efficiency as soon as the function $f$ satisfies regularity assumptions. Closely related are nearest-neighbor approaches, which have been studied by several authors (see, e.g., \cite{devroye2003estimation,liitiainen2008nonparametric,liitiainen2010residual,devroye2013strong,gadat2016classification,gyorfi2015asymptotic,devroye2018nearest}). For instance, in \cite{devroye2018nearest}, the authors propose a plug-in estimator with statistical consistency for any $d$, and a central limit theorem with rate $\sqrt{n}$ for $d\leq 3$, provided that regularity assumptions hold. In parallel, \cite{broto2020variance} consider a variant that is consistent for any $d$, but no rate of convergence is provided. When $d=1$, a central limit theorem for estimators based on ranks (i.e., nearest neighbors on the right) is also proved in \cite{GGKL20}.
Recently, in \cite{da_veiga_efficient_2024}, the authors used high-order kernels to build asymptotically efficient estimators of Sobol' indices of any dimension from a single i.i.d.\ $n$ input/output sample. Their result requires strong smoothness assumptions on $f$ and on the support of the input variables {which may be restrictive in some irregular situations.}
\medskip

In order to build an estimator of Sobol' indices using any of the methods described above, the practitioner needs to have access to the complete dataset. Unfortunately, in some applications, the data are generated online and cannot be stored.%
The goal of this work is to introduce and to  study a new numerical scheme based on online data for simultaneously estimating the $2^p$ Sobol' indices taking into account their geometrical constraints.

\subsection{Organisation of the paper}

The article is organized as follows. Section \ref{sec:sobol_intro} introduces a variational characterization of the collection of all Sobol' indices as the minimizer of a strongly convex function (Corollary \ref{cor:variational_sobol}). This section also summarizes well-known key concepts from global sensitivity analysis. Section \ref{Sec:sgd_sobol} presents the Pick–Freeze Mirror algorithm, proposed to solve the constrained optimization problem and to construct a sequence $\Sbf^*_n$ of estimators for the Sobol' indices. Section \ref{sec:experiments} then presents numerical experiments that illustrate the performance and practical behavior of the proposed algorithm on benchmark problems. Finally, Section \ref{sec:proofs} is devoted to the proofs of our theoretical results.

\section{Constrained variational characterization of Sobol' indices}
\label{sec:sobol_intro}

\subsection{Framework and notations}\label{sec:framework}
{We introduce below a set of key definitions and notations, which are highlighted in \textbf{bold} for improved readability.}

{\paragraph{Numerical code}
We consider $X = (X^1, \ldots, X^p)$, a real-valued $p$-dimensional random vector representing the input of a numerical code, and $Y$, the output random variable associated with a black-box function $f: \R^p \longrightarrow \R$. The output $Y$ is linked to the input $X$ through the simple relation
\begin{equation}
\label{eq:mod}
Y = f(X) \quad \text{where} \quad X = (X^1, \ldots, X^p).
\end{equation}
We assume that $X$ follows an unknown distribution $\QQ$, and that its components $(X^j)_{j \le p}$ are \textit{independent}. For numerical purposes, we further assume that the variables $(X^j)_{j \le p}$ can be easily simulated.
Our theoretical results require only that
\begin{equation}
\label{eq:moment4}
\E[Y^4] < \infty.
\end{equation}}

{\paragraph{Subset notations}
Throughout the paper, we frequently consider subsets of variables, denoted by $u \subset \{1, \ldots, p\} \setminus \emptyset$. For convenience, we denote by $\Ups$ the set of \textit{all} subsets of $\{1, \ldots, p\}$.
When $u \in \Upsilon$, $|u|$ denotes the cardinality of the subset $u$, with the convention that $|\emptyset| = 0$. {Below, $q=2^p$ will refer to the cardinal of $\Upsilon$.}}
  
{
\paragraph{Sobol' indices}
A now-standard approach for global sensitivity analysis relies on the computation of Sobol' indices, a technique introduced in \cite{pearson1915partial} and later popularized by \cite{sobol2001global}. These indices are particularly suitable when the numerical code returns a real-valued output, denoted by $Y$.}

{Using the ANOVA-Hoeffding decomposition of the variance (see \textit{e.g.} \cite{Hoeffding48,Sobol1969,efronstein}), the objective of Sobol' indices is to compare the conditional variance of $Y$, when fixing the values of a subset of variables
$(X^i, i \in u)$, to the total variance of $Y$
\begin{equation}
\label{eq:Hoeff}
\Var(Y)=\sum_{u \in \Ups} V_u
\quad \text{
with }
V_u=\sum_{v\subset u}(-1)^{|u|-|v|}\Var(\E[Y|X^i, i\in v]) \ge 0.
\end{equation}
We shall observe that $V_\emptyset= \Var(\E[Y]) = 0$, regardless the distribution of the inputs $X$, so that $V_\emptyset$ is always zero and does not need any estimation procedure.
The Sobol' $\Sb=(S^*_u)_{u\in \Ups}$ and closed Sobol' $\mathbf{\Sc}=(\Sc_u)_{u\in \Ups}$ indices with respect to $(X^i,i\in u)$ are defined by:
\begin{equation}\label{def:sobol}
S^*_u:=\frac{V_u}{\Var(Y)} \, 
\quad \text{and} \quad
\Sc_u:=\frac{\Var(\E[Y|X^i, i\in u])}{Var(Y)}
\,.
\end{equation}

{\paragraph{Simplex $\Delta_q$}
Dividing both sides of \eqref{eq:Hoeff} by $\Var(Y)$, we observe that $\Sb$ belongs to the simplex of discrete probability distributions over subsets defined by: \[\Delta_q:=\left\{s=(s_u)_{u\in \Ups }, \forall u \in \Ups: s_u \ge 0 \quad \text{and} \quad s_{\emptyset}=0 \quad \text{and} \quad \sum_{u\in \Ups} s_u=1.\right\}.
\]
Moreover, we also deduce from Equation \eqref{eq:Hoeff} that Sobol' and closed Sobol' indices are linked through a linear transformation (see Appendix \ref{app:spectre-M} for more details). Hence, we have:
\begin{equation}
    \label{def:M}
    \Sb=\Mb\mathbf{\Sc} \quad \text{with} \quad 
    \forall (u,v) \in \Ups \times \Ups \qquad \Mb_{u,v}=  (-1)^{|u|-|v|} \mathbf{1}_{v \subset u}.
\end{equation}
Finally, the matrix $\Mb$ can be inverted using the Rota–Möbius inversion formula \cite{rot64}, thus

\begin{equation}
    \label{def:M_inv}
    \forall (u,v) \in \Ups \times \Ups \qquad \Mb^{-1}_{u,v}=  \mathbf{1}_{v \subset u} \quad \text{and} \quad 
    \mathbf{\Sc} = \Mb^{-1} \Sb.
\end{equation}}

\subsection{Pick-Freeze trick}

In this section, we briefly recall the classical "Pick-Freeze" strategy for estimating Sobol' indices.
We consider  $X'=(X'^1,\ldots,X'^p)$ an independent copy of $X=(X^1,\ldots,X^p)$ so that $(X,X') \sim \QQ \otimes \QQ$.
\begin{prop}[Pick-Freeze trick]\label{prop:PickFreeze}
\label{cor:variational_sobol}
For any $u\in \Ups$, we have:
\begin{itemize}
    \item[(i)] If $Y=f(X)$, $\Yu=f(X^{(u)})$ where $X^{(u)}$ is built from $X$ and $X'$ as 
$$
\forall i \in \{1,\ldots,p\} \quad X^{(u),i}=X^i \mathbf{1}_{i \in u} + X'^i \mathbf{1}_{i \notin u},
$$
then $\Var(\E[Y|X^i, i\in u])=\Cov(Y,\Yu) $ and 
\begin{equation}
\label{eq:PF2}
\Sc_u=\frac{\Cov(Y,\Yu)}{\Var(Y)}.
\end{equation}
In what follows, $(Y,Y^u)$ will be refered to as a Pick-Freeze sample.
\item[(ii)] Furthermore $$\Sc_u = \arg\min_{\theta \in [0,1]} \psi_u(\theta) \quad \text{with} \quad \psi_u(\theta) := \E\left[\left((Y-\E[Y]) \theta -(\Yu-\E[Y])\right)^2\right],
$$
where  the expectation has to be considered with respect to the pair of random variables $(Y,\Yu)$ defined above.
\end{itemize}
\end{prop}
For a fixed subset $u \in \Ups$, the ``Pick–Freeze" trick leads to the so-called ``Pick–Freeze" estimator, which corresponds to a Monte Carlo version of definition \eqref{eq:PF2}. These estimators have been extensively studied in the literature \cite{sobol2001global,saltelli2008global,Monod2006} and are known to be consistent and asymptotically normal \cite{janon2012asymptotic}. In addition to their favorable asymptotic properties, it is worth noting that the number of evaluations of $f$ required to estimate $\Sc_u$ is $2n$, yielding a convergence rate of $\sqrt{n}$. To estimate $K$ Sobol' indices, this number increases to $(K+1)n$, so that the total computational cost reaches $2^p n$, which grows exponentially with $p$.

\subsection{Variational characterization of $\Sc$ and $\Sb$}\label{sec:caracterisation}
In this work, we propose a method for the simultaneous estimation of \textit{all} Sobol' indices, assuming access to a sequence of data compatible with the Pick–Freeze trick.
The variational characterization of $\mathbf{\Sc}$, stated in Proposition \ref{prop:PickFreeze} $(ii)$, implies that a standard gradient descent method can be naturally applied to obtain an accurate approximation of $\Sc_u$, since the functional $\theta \mapsto \psi_u(\theta)$ is strongly convex.

We perform a convex aggregation of all the functions $\psi_u$ using a collection of weights $(a_u)_{u \in \Ups}\in\Delta_q$, for now chosen arbitrarily, and define the function $\Psi^a$ as follows:
\begin{equation}
    \label{def:psi-a}
    \forall x =(x_u)_{u \in \Ups} \qquad 
    \Psi^{a}(x) = \frac{1}{2} \sum_{u \in \Ups} a_u \psi_u(x_u).
\end{equation}

The probability distribution $(a_u)_{u\in\Ups}$ which will be denoted by $\mathcal{L}^a_{\Ups}$ is for now chosen arbitrarily such that:
\begin{equation}
\label{eq:hyp_a}
\min_{u\in\Ups} a_u>0.
\end{equation}
For any such $a$, a straightforward consequence of Proposition \ref{cor:variational_sobol} $(ii)$ is that $\Psi^a$ is minimized for $x=\mathbf{\Sc}$.

Finally, recall that $\Sb=\Mb \mathbf{\Sc}$ and $\Sb\in\Delta_q$,
and define
\begin{equation}
    \label{def:Phi-a}
    \forall s \in \Delta_q: \qquad \Phi^{a}(s):=\Psi^{a}(\Mbm s).
\end{equation}
This leads to the following characterization result for the set of \textit{all} Sobol' indices, which will serve as the cornerstone of our work.
\begin{cor}[Variational characterization of $\Sb$]
\label{theo:variational_sobol}
For any discrete positive probability distribution $a=(a_u)_{u \in \Ups}$, satisfying \eqref{eq:hyp_a}, the collection of Sobol' indices $\Sb$ is uniquely defined as:
\begin{equation}
    \label{eq:caracterisation_sb}
    \Sb=\arg\min_{s\in \Delta_q } \Phi^{a}(s), 
\end{equation}
where $\Phi^{a}$ can be written as an strongly convex  function represented as an expectation:
\begin{equation}
    \label{def:Phi-a_esp}\Phi^{a}(s)= \E\left[\left((Y-\E[Y]) [\Mbm s]_U -(Y^{(U)}-\E[Y])\right)^2\right],
\end{equation}
where $U \sim \mathcal{L}^a_{\Ups}$ and conditionally to $U$, $(Y,Y^U)$ is a Pick-Freeze sample (see Proposition \ref{prop:PickFreeze}).
\end{cor}
Corollary \ref{theo:variational_sobol} shows that $\Sb$ minimizes the function $\Phi^a$, which is strongly convex and differentiable, over   $\Delta_q$. Moreover, $\Phi^a$ is the expectation of a strongly convex function that can be readily simulated using Pick-Freeze realizations. Consequently, the problem of estimating all Sobol' indices can be addressed via a constrained stochastic optimization algorithm. This is precisely the approach we adopt in the following, employing a Stochastic Mirror Descent algorithm detailed in the next section.

\section{Online estimation of Sobol' indices with Pick-Freeze S.M.D.}
\label{Sec:sgd_sobol}
Our estimation strategy is based on a mirror descent introduced in the pioneering work of \cite{NY83}; it provides a versatile approach that naturally handles constrained optimization problems through an appropriate mirror map. Mirror descent can be seen as a natural generalization of the gradient descent method, formalizable as an iterative Maximization-Minimization procedure. It is particularly suited for the smooth minimization of a strongly convex function $\Phi^a$ over the constrained set $\Delta_q$ (see, \textit{e.g.}, \cite{Lan:2012}, \cite{Bubeck:2015}). The mirror descent algorithm defines a smooth trajectory that remains within the constrained set without requiring an additional projection step and effectively "pushes" the boundaries of the simplex to an infinite distance from any point strictly inside $\Delta_q$. Moreover, it can be easily extended to a \textit{stochastic} setting using iterative noisy realizations of the gradients.

\subsection{Bregman divergence on $\Delta_q$ and  mirror descent}
We first introduce the strongly convex negative entropy function over $\Delta_q$:
\begin{equation}\label{def:phi}
\forall v \in \Delta_q: \qquad h(v) := \sum_{i=1}^q v_i \log v_i.
\end{equation}
If $\langle \cdot , \cdot \rangle$ denotes the standard Euclidean inner product, the $h$-Bregman divergence $\Dp$ between two probability distributions is defined as
\begin{equation*}\label{def:Dphi}
\forall (w,v) \in \Delta_q^2, \qquad \Dp(w,v) := h(w) - h(v) - \langle \nabla h(v), w - v \rangle.
\end{equation*}

\begin{rem}
 A Bregman divergence induces a natural metric associated to $u \in \Delta_q$ while we emphasize that the strong convexity of $h$ induces the following lower bound:
 \begin{equation}
     \label{eq:Dp_lower}
      \Dp(w,v) \ge \frac{\|w-v\|^2}{2}.
 \end{equation}
\end{rem}
A (deterministic) mirror descent with a step-size sequence $(\gamma_n)_{n \ge 1}$ consists of minimizing, from $n$ to $n+1$, the first-order Taylor approximation of $\Phi^a$ penalized by the Bregman divergence. The iterative step corresponds to a sequence $(\theta_n)_{n \ge 1}$:
\begin{equation}\label{def:MD}
\theta_{n+1} = \arg \min_{\theta \in \Delta_q} \left\{ \Phi^a(\theta_n) + \langle \nabla \Phi^a(\theta_n), \theta - \theta_n \rangle + \frac{1}{\gamma_{n+1}} \Dp(\theta, \theta_n) \right\}.
\end{equation}
Equation \eqref{def:MD} defines a general ``proximal'' descent method built using the Bregman divergence $\Dp$ on the function $\Phi^a$. When $\Dp$ is replaced by the standard Euclidean distance, this reduces to the standard gradient descent with step-size sequence $(\gamma_n)_{n \ge 1}$.

{
Since the function 
$
\theta \longmapsto \Phi^a(\theta_n) + \langle \nabla \Phi^a(\theta_n), \theta - \theta_n \rangle + \frac{1}{\gamma_{n+1}} \Dp(\theta, \theta_n)
$
is strictly convex, the iterate $\theta_{n+1}$ is characterized by the first-order optimality condition. Using the identity 
$
\nabla_\theta \Dp(\theta, \theta_n) = \nabla h(\theta) - \nabla h(\theta_n),
$
we obtain that $\theta_{n+1}$ satisfies
\[
\nabla \Phi^a(\theta_n) + \frac{1}{\gamma_{n+1}}\left(\nabla h(\theta_{n+1}) - \nabla h(\theta_n)\right) = 0.
\]
Rearranging the terms yields
\[
\nabla h(\theta_{n+1}) = \nabla h(\theta_n) - \gamma_{n+1}\nabla \Phi^a(\theta_n).
\]
Using the conjugate properties of the convex function $h$ (Fenchel--Legendre transform), this relation can be rewritten as
\[
\theta_{n+1}
= \nabla h^{-1}\left(\nabla h(\theta_n) - \gamma_{n+1}\nabla \Phi^a(\theta_n)\right)
= \nabla h^*\left(\nabla h(\theta_n) - \gamma_{n+1}\nabla \Phi^a(\theta_n)\right).
\]
In the context of our $D_h$-penalized minimization problem, the mappings $\nabla h$ and $\nabla h^*$, which are inverse to each other, provide a natural interpretation of the algorithm: a gradient step is first performed in the dual space (i.e., the image of $\nabla h$), and the new iterate $\theta_{n+1}$ is then obtained by mapping back to the primal space through $\nabla h^*$.}

{
In our setting, where $h$ is given by the entropy defined in \eqref{def:phi}, a remarkable property is that $\nabla h^*$ admits an explicit expression; see \textit{e.g.} \cite{Lan:2012,CGH24}. This allows us to obtain the following explicit update:
\begin{equation}
    \label{eq:MD_explicit}
    \theta_{n+1} = \frac{\theta_n e^{-\gamma_{n+1} \nabla \Phi^a(\theta_n)}}{\|\theta_n e^{-\gamma_{n+1} \nabla \Phi^a(\theta_n)}\|_1},
\end{equation}
where the above relation is understood componentwise in $\mathbb{R}^q$. We refer to \cite{Lan:2012} for further details.}

\subsection{The Pick-Freeze Mirror algorithm}
An important feature of the definition of $\Phi^a$ is that this function can be expressed as an expectation over a triplet of random variables $(U,Y,Y^U)$, which are assumed to be easily sampled in our context.
To design the Pick-Freeze Mirror algorithm, we begin with the computation of $\nabla \Phi^a$.
\begin{prop}[Computation of $\nabla\Phi^a$]
\label{prop:gradient_Phi_a}
\label{prop:gradient_estimator_unbiased}
Assume that $U$ is a random subset sampled according to $\mathcal{L}^a_{\Ups}$ and conditionally on $U$, $(Y,Y^U)$ is a Pick-Freeze sample, then for any $s \in \Delta_q$, one has:
\begin{align*}
\nabla \Phi^{a}(s)&=
 \nabla_s \mathbb{E}[\phi(s,Y,Y^U,U)] = \mathbb{E}\left[\nabla_s \phi(s,Y,Y^U,U)\right]
 \end{align*}
 where $\phi(s,Y,Y^U,U)$ is the random variable defined by:
 \begin{equation}\label{def:phi-sto}
\nabla \phi(s,Y,Y^U,U):= (Y-\E[Y]) M^{-1}_{U,:} \left((Y-\E[Y])[M^{-1}s]_U-(Y^U-\E[Y]) \right).
\end{equation}

\end{prop}
{
The first contribution of the previous proposition is to provide an exact expression for the gradient of the function $\Phi^a$ for any value of $s$. In principle, it would also be possible to use this expression to implement a gradient descent method for the functional $\Phi^a$. However, a careful inspection of the gradient formula reveals that it is—or would be—necessary to perform $2^p$ Pick-Freeze evaluations by enumerating all subsets 
$U$, each coupled with an additional Monte-Carlo routine to approximate the expectation of each random variable. Since this strategy becomes numerically impractical as soon as the number of variables  $p$ exceeds 10, it is therefore necessary, in order to define an effective algorithm, to replace both the Monte Carlo step and the enumeration of subsets with a stochastic simulation strategy. This is made possible by noting that the gradient can be written as an expectation with respect to $U, Y$ and $Y^U$. 
Furthermore, in Algorithm 1 (see below), we use only a single evaluation of $U, Y$ and $Y^U$
 to efficiently update our Sobol index estimates. In parallel with this cheap gradient strategy, we also leverage the simplex constraint to update all coordinates of the current estimate with just one Pick-Freeze evaluation.
}

In our setting, an additional challenge arises from the fact that $\mathbb{E}[Y]$, which is involved in the computation of $\nabla \Phi^a$ via Equation \eqref{def:phi-sto}, is unknown and must be estimated online. Consequently, we introduce a biased Stochastic Mirror Descent algorithm, which jointly estimates both the mean of $Y$ and all Sobol' indices. This biased framework is similar to that used in \cite{CGH24}.
\paragraph{Description of the algorithm with unknown mean}

We assume that we have at our disposal a sequence of i.i.d. random variables $(\Xn,\Xpn,\Un)_{n \ge 1}$, where $(\Xn,\Xpn,\Un) \sim \QQ \otimes \QQ \otimes \mathcal{L}^a_{\Ups}$, and $\mathcal{L}^a_{\Ups}$ denotes the discrete distribution over $\Ups$ such that
$$
\forall n \ge 1 \quad \forall u \in \Ups \qquad 
\mathbb{P}\left( \Un = u\right) = a_u.
$$
We can therefore compute, for each $n \ge 1$, a Pick-Freeze sample $(\Yn,\YUn) = (f(\Xn), f(\Xn^{U_n}))$, as stated in Corollary \ref{theo:variational_sobol}. In particular, Proposition \ref{prop:gradient_estimator_unbiased} shows that ${\nabla}\phi(s,\Yn,\YUn,\Un)$ is an unbiased estimator of the gradient $\nabla \Phi^a$ for any $s \in \Delta_q$.
We estimate $\mathbb{E}[Y]$ using the empirical mean of the samples $(\Yn)$. Let $\mn$ be defined by

\[ \mn=\frac{1}{n}\sum_{k=1}^n Y_k.
\]

Then, based on Proposition \ref{prop:gradient_estimator_unbiased} and, specifically, Equation \eqref{def:phi-sto}, a recursive estimator can be derived by substituting the unknown mean with its empirical counterpart $\mn$:
\begin{equation}
\label{eq:gradient_estimator}
\widehat{\nabla \phi}_{n+1}(s) =
(\Ynp - \mn) M^{-1}_{\Unp,:} \left[ (\Ynp - \mn) [M^{-1}s]_{\Unp} - (\YUnp - \mn) \right].
\end{equation}

We emphasize that in Equation \eqref{eq:gradient_estimator} above, we use $\mn$ and not $\mnp$ to compute $\widehat{\nabla \phi}_{n+1}$, for the sake of the simplicity from a mathematical point of view.
\medskip

Finally we can define the recursive sequence $(\Sbn)_{n\ge0}$ as:
\begin{equation} 
\label{eq:MD}
\Sbnp= \arg \min_{s \in \Delta_q } \left\{ \langle 
\widehat{\nabla \phi}_{n+1}(\Sbn),s -\Sbn \rangle + \frac{1}{\eta_{n+1}} \Dp(s,\Sbn)\right\} = \frac{\Sbn e^{-\eta_{n+1} \widehat{\nabla\phi}_{n+1}(\Sbn) }} {\|\Sbn e^{-\eta_{n+1} \widehat{\nabla \phi}_{n+1}(\Sbn)}\|_1}
.
\end{equation}

The Pick-Freeze Mirror algorithm then derives from a sequence of iterations that respectively computes $\mn$, samples $(\Yn,\YUn,\Un)$ and uses \eqref{eq:MD}. This new method is summarized in Algorithm
\ref{algo:Black-Mirror}.

\begin{algorithm}[h!]
\caption{Pick-Freeze Mirror algorithm}\label{algo:Black-Mirror}
\begin{algorithmic}
\REQUIRE{ $N \geq 0$, Step-size Sequence $(\gamma_n)_{n \ge 1}$}
\STATE{ $\hat{m}_0 \gets 0$}
\STATE{ $\hat{\mathbf{S}}^\emptyset_0 \gets 0$ }
\STATE{ $\hat{\mathbf{S}}^*_0 \gets \frac{\mathbf{1}_{q-1}}{q-1}$}\hfill\COMMENT{  Initialization with a uniform distribution of weights over $\Delta_q$}
\FOR{$n \gets 1$ to $N$}
 \STATE{ Sample $(\Xnp,\Xpnp,\Unp) \sim \QQ \otimes \QQ \otimes \mathcal{L}^a_{\Ups} $ and compute $\Ynp,\YUnp$}
 \STATE{ Update the \textbf{estimator} of the expectation of $Y$:
 $$
 \mnp = \mn + \frac{1}{n+1}(\Ynp-\mn)
$$ }
\hfill\COMMENT{Empirical mean estimation}
\STATE{ Compute the estimator of  $\nabla \Phi^a$ at step $n$:
$$
\Hhnp(\Sbn)=
        (\Ynp-\mn) M^{-1}_{\Unp,:} \left[(\Ynp-\mn)[M^{-1}\Sbn]_{\Unp}-(\YUnp-\mn)\right]
 $$}
\hfill \COMMENT{ Gradient estimation}
 \STATE{ Update the estimator of the collection of Sobol' indices:
$$
\Sbnp = 
\frac{\Sbn e^{-\eta_{n+1} \widehat{\nabla\phi}_{n+1}(\Sbn) }} {\|\Sbn e^{-\eta_{n+1} \widehat{\nabla \phi}_{n+1}(\Sbn)}\|_1}$$}
\hfill\COMMENT{Stochastic Mirror Descent}
\ENDFOR
\end{algorithmic}
\end{algorithm}

\subsection{Theoretical results}
We state below our theoretical results.

\paragraph{Almost sure convergence}

Theorem \ref{theo:cv_ps}  establishes the convergence of $(\Sbn)_{n \ge 1}$ to the set of Sobol' indices $\Sb$ under weak assumptions on the sampling process (finite fourth moment of $Y$) and on the step-size sequence. In particular, our result does not require any smoothness assumptions on the numerical code $f$, unlike recent contributions \cite{da2009local,da2008efficient,plischke2020fighting,solis2021non,heredia2021nonparametric}, which use a kernel-based approach that inherently imposes regularity conditions on the numerical code.
\begin{theo}[Almost sure convergence of $(\Sbn)_{n \ge 0}$  \label{theo:cv_ps}]
Assume that the sequence $(\eta_n)_{n \ge 1}$ satisfies:
\begin{equation}
    \label{hyp:pas}
\sum\eta_{n+1}=\infty,\qquad \text{and} \qquad \sum \eta_{n+1}^2<\infty.
\end{equation}
Assume furthermore that $Y$ has a 4-th order moment and that  the discrete probability distribution $a$ is lower bounded:
$
\min_{u \in \Ups} a_u >0,
$
{and that all the coordinates of 
$\hat{\mathbf{S}}_0$ are strictly positive, 
then $(\Sbn)_{n \ge 1}$ almost surely converges towards  $\Sb$}.
\end{theo}

These results hold for a standard setup in stochastic algorithms with a decreasing learning rate $(\eta_n)_{n \ge 1}$. The key condition is the convergence of the series $\sum_{n \ge 1} \eta_n^2$ together with the divergence of $\sum_{n \ge 1} \eta_n$. A typical application of Theorem \ref{theo:cv_ps} corresponds to $\eta_n = \eta_0 n^{-\alpha}$ with $\alpha \in (1/2,1]$. 
{The previous result guarantees the convergence of the algorithm for any initialization, under the assumption that none of the coordinates is equal to zero. In this sense, the behavior of the algorithm is therefore insensitive to the initialization over an infinite (asymptotic) horizon. By contrast, the influence of the initialization appears in the non-asymptotic bounds given in our next result.}

\paragraph{Non-asymptotic upper bound}
Using the perspective introduced in \cite{moulines2011non,doi:10.1137/120880811} to assess the computational cost of convex stochastic optimization, it is possible to derive a more quantitative result on the sequence $(\Sbn)_{n \ge 1}$. This result is expressed in terms of the expected value of $\Phi^a$ throughout the algorithm.
For this purpose, we introduce below a quantity related to the discrete probability distribution used for sampling the subsets $U$ according to $a$:
\begin{equation}
|a|_{\exp,\ell} := \mathbb{E}\left[ 2^{\ell |U|} \right] = \sum_{u \in \Ups} a_u 2^{\ell |u|}.
\label{eq:norme-a-power}
\end{equation}

\begin{theo}[Non-asymptotic risk bound\label{theo:sobol_non_asympt}]
 Consider any finite horizon $n>0$ and the sequence $(\bar{\Sbf}^{\eta}_{k})_{k \ge 1}$, the Cesaro averaged sequence built from the sequence $(\Sbk)_{k \ge 1}$:
 $$
\bar{\Sbf}^{\eta}_{n} = \frac{\sum_{k=1}^n \eta_{k+1} \Sbk}{\sum_{k=1}^n \eta_{k+1}}.
$$
\begin{itemize}
    \item[$i)$] Assume that $(\eta_{k+1})_{k\ge0}$ is piecewise constant defined by:
$$ \eta_{k+1}=\frac{\mathbf{1}_{k \leq n-1}}{\sqrt{(1+\sqrt{\|a\|_{\exp,2}}) n}},$$
then
\begin{align*}
 \E[\Phi^a(\bar{\Sbf}^{\eta}_{n})]&-\min \Phi^a \\
 &\leq \sqrt{ 1+\sqrt{\|a\|_{\exp,2}}} \left( \Dp(\Sb,\Sbun)  + C \E[Y^4] \right) \frac{1}{\sqrt{n}} + 5\|a\|_{\exp} \Var(Y) \frac{\log(n)}{n}.
\end{align*}
\item[$ii)$] If $(\eta_{k+1})_{k\ge0}$ is horizon free (that does not depend on the final iteration $n$) defined by $
\eta_{k+1} = \eta_0 (k+1)^{-1/2},
$
then $\forall n\in\N$:
\begin{align*}
\E[\Phi^a(\bar{\Sbf}^{\eta}_{n})]&-\min \Phi^a
\leq C \left( \Dp(\Sb,\Sbun)  + \|a\|_{\exp,2} +  \E[Y^4] \right) \frac{\log n}{\sqrt{n}}.
\end{align*}
\end{itemize}
\end{theo}
 
\section{Experimental results \label{sec:experiments}}
In this section, we explore the numerical performances of our algorithm. 
{Our numerical experiments are conducted solely on simulated scenarios, providing preliminary evidence of the practical effectiveness of our method. In future work, we plan to benchmark the approach on realistic real-world datasets.}
{A python code is accessible online and contains the basic functions to run our algorithm and a notebook with examples \footnote{\url{https://plmlab.math.cnrs.fr/costa2150/mirror-pick-freeze/}}.}
\subsection{Baselines and Competing Methods}
Our study focuses on comparing our method with two existing estimation techniques for Sobol' indices: the standard Pick-Freeze estimator for individual Sobol' indices \cite{janon2012asymptotic}, and the kernel-based estimator of \cite{da_veiga_efficient_2024}. We consider both regular and irregular functions $f$. Let us first briefly discuss the advantages and limitations of these two methods.

\begin{itemize}
    \item The asymptotically efficient version of the Pick-Freeze estimator for closed Sobol' indices studied in \cite{janon2012asymptotic} (referred to as "PF1" below). For each subset $u$, it computes a Monte Carlo estimator $T_{u,N}$ of $\Sigma_u$ based on a Pick-Freeze sample  $(Y_i,Y_i^{(u)})_{1\le i\le N}$ :
    $$T_{u,N} = \frac{\frac{1}{N}\displaystyle\sum_{i=1}^N Y_iY_i^{(u)}- \left(\frac{1}{N}\displaystyle\sum_{i=1}^N \left[\frac{Y_i+Y_i^{(u)}}{2}\right]\right)^2}{\frac{1}{N}\displaystyle\sum_{i=1}^N \frac{Y_i^2+(Y_i^{(u)})^2}{2}- \left(\frac{1}{N}\displaystyle\sum_{i=1}^N \left[\frac{Y_i+Y_i^{(u)}}{2}\right]\right)^2}$$
   \begin{itemize} \item Advantages: The Pick-Freeze estimators are consistent and satisfy a (joint) central limit theorem under the sole assumption that the output has a finite fourth-order moment.
\item Drawbacks: To estimate all $2^{p-1}$ Sobol' indices at a $\sqrt{n}$ rate, a sample of size $2^p n$ is required.
\end{itemize}
    \item The kernel-based estimator given in equation (20) of \cite{da_veiga_efficient_2024}, using an Epanechnikov kernel of order 2 and 4 (see [19] for a definition), with the kernel bandwidth optimized via leave-one-out on the regression function ("Kernel 2" and "Kernel 4"). 
  \begin{itemize}
        \item Advantages: A simple i.i.d. sample of size $n$ is theoretically sufficient to estimate all Sobol' indices at the $\sqrt{n}$ rate.
\item Drawbacks: The input variables are assumed to be compactly supported and absolutely continuous with respect to the Lebesgue measure. Moreover, regularity assumptions on the regression function are also required. Although asymptotic normality is established for all Sobol' indices, the computational cost of evaluating the estimator appears to grow exponentially with the order of the indices.
    \end{itemize}
\end{itemize}
For all these techniques, 50 replicates of the estimators are computed up to step $N_0 = 500$.\
We compare these methods with our Pick-Freeze Mirror Descent algorithm (Algorithm \ref{algo:Black-Mirror}) using a decreasing step sequence $\eta_n = \eta_0 / \sqrt{n+1}$, where $\eta_0$ is a tuning parameter. To fairly compare the results and account for the fact that the Pick-Freeze Mirror algorithm computes all indices simultaneously, we consider the algorithm at times $N_0$ and $2^p N_0$, corresponding respectively to "Mirror$\_$PF-1" and "Mirror$\_$PF-2".

\subsection{Numerical code used and comparison}
\subsubsection{Case of a smooth regular function: the Bratley function}
We first consider the Bratley function, defined as
\[
f_{\text{Bratley}}(X^1, \ldots, X^p) =
\sum_{i=1}^p (-1)^i \prod_{j=1}^i X^j,
\]
where $X^i \sim \mathcal{U}([0, 1])$ are i.i.d. and $p = 5$.
For this function, \cite{da_veiga_efficient_2024} considered the performance of their algorithm for the estimation of both the first-order and total-order indices. Recall that the first-order indices are the Sobol' indices associated with singletons, while the total-order indices are defined as
\[
S^{\text{tot}}_i = 1 - \Sigma^*_{\{1, \ldots, p\} \setminus \{i\}}.
\]
In particular, since the total-order indices depend on the closed Sobol' indices, we estimate them using our sequence $(\Sbn)_{n \ge 1}$ and \eqref{def:M}.

\begin{figure}[h!]
    \centering
    \includegraphics[scale=0.45]{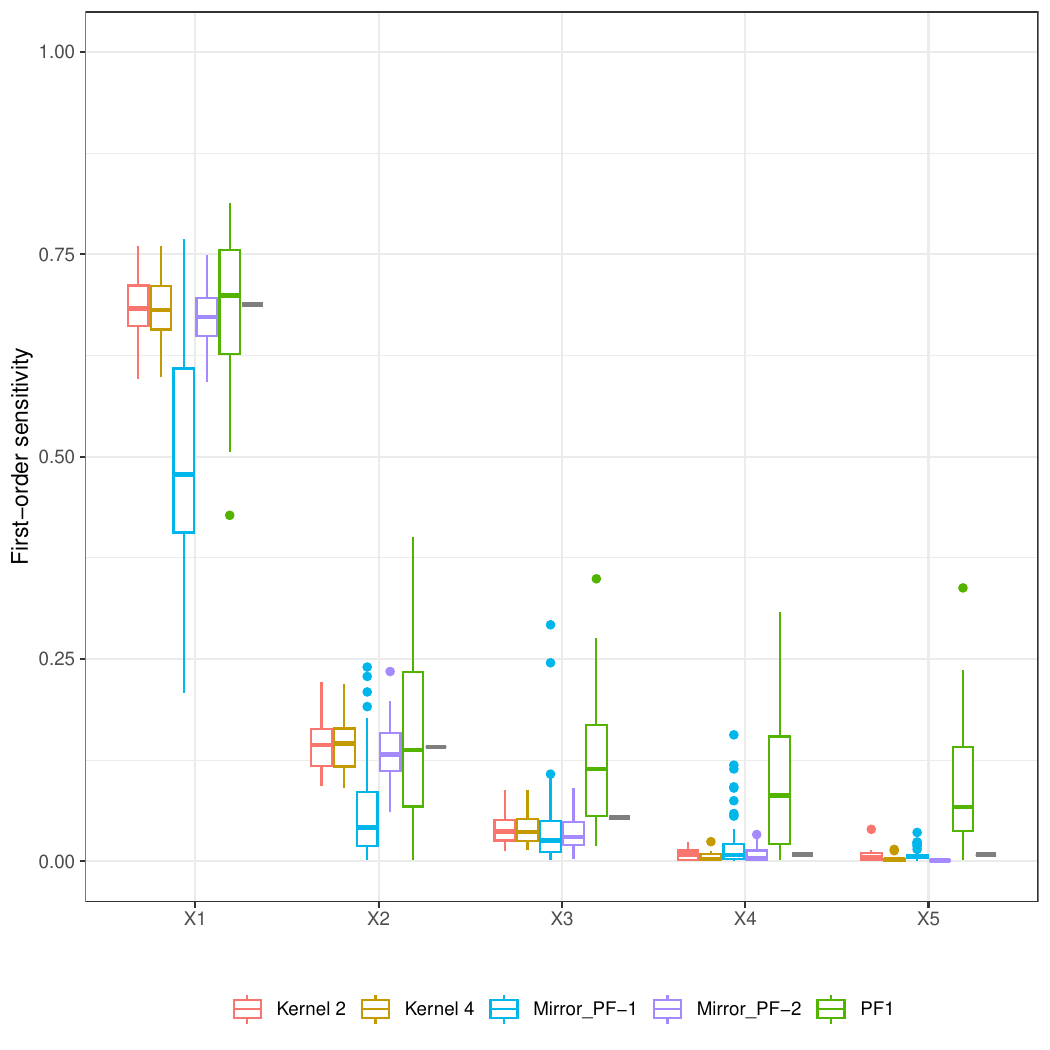}
    \includegraphics[scale=0.45]{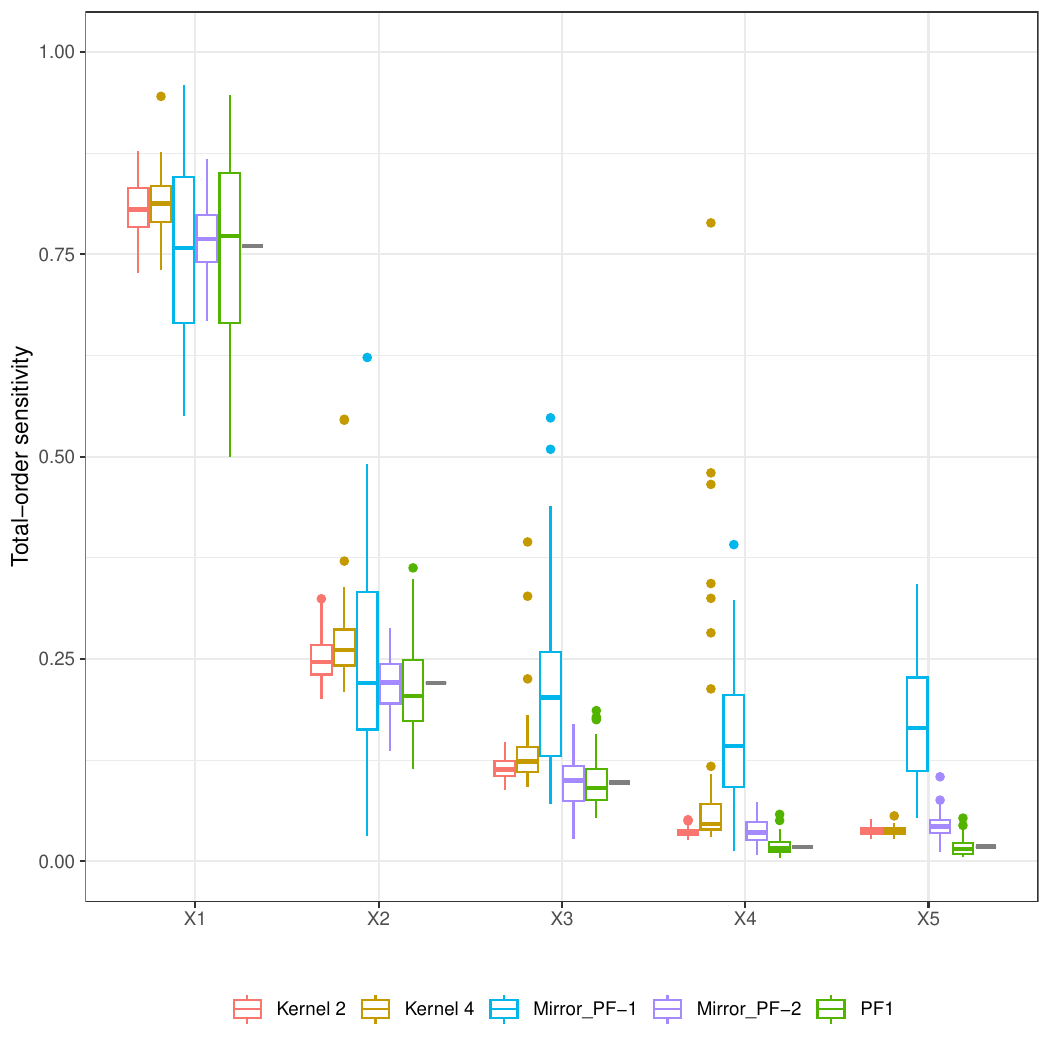}
    \caption{Estimators of the first-order indices (left panel) and total-order indices (right panel) of the Bratley function obtained with different methods. The reference value is shown as a gray line.}
    \label{fig:sobol_bratley}
\end{figure}

\subsubsection{Case of a discontinuous function}
We then consider a discontinuous function defined by 
$$f_{disc}(X^{(1)},X^{(2)},X^{(3)}) = X^{(1)}\un_{X^{(3)}<0}+ (X^{(2)})^2 \un_{X^{(3)}\ge0} + X^{(3)},$$
for $X^i$ i.i.d $\mathcal N(0,1)$ random variables. 
This case is interesting because it does not satisfy the regularity assumptions required for the kernel based methods, nor the assumption on compactly supported random variables.

\begin{figure}[h!]
    \centering
    \includegraphics[width=0.45\linewidth]{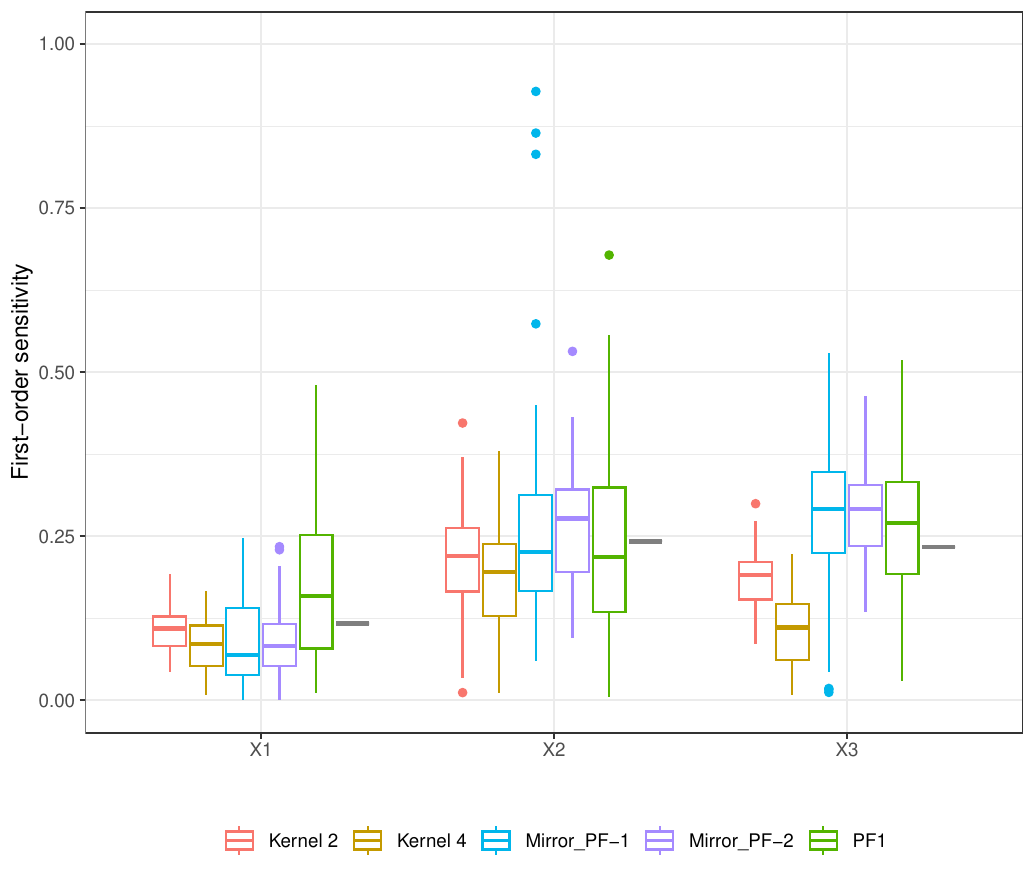} 
    \includegraphics[width=0.45\linewidth]{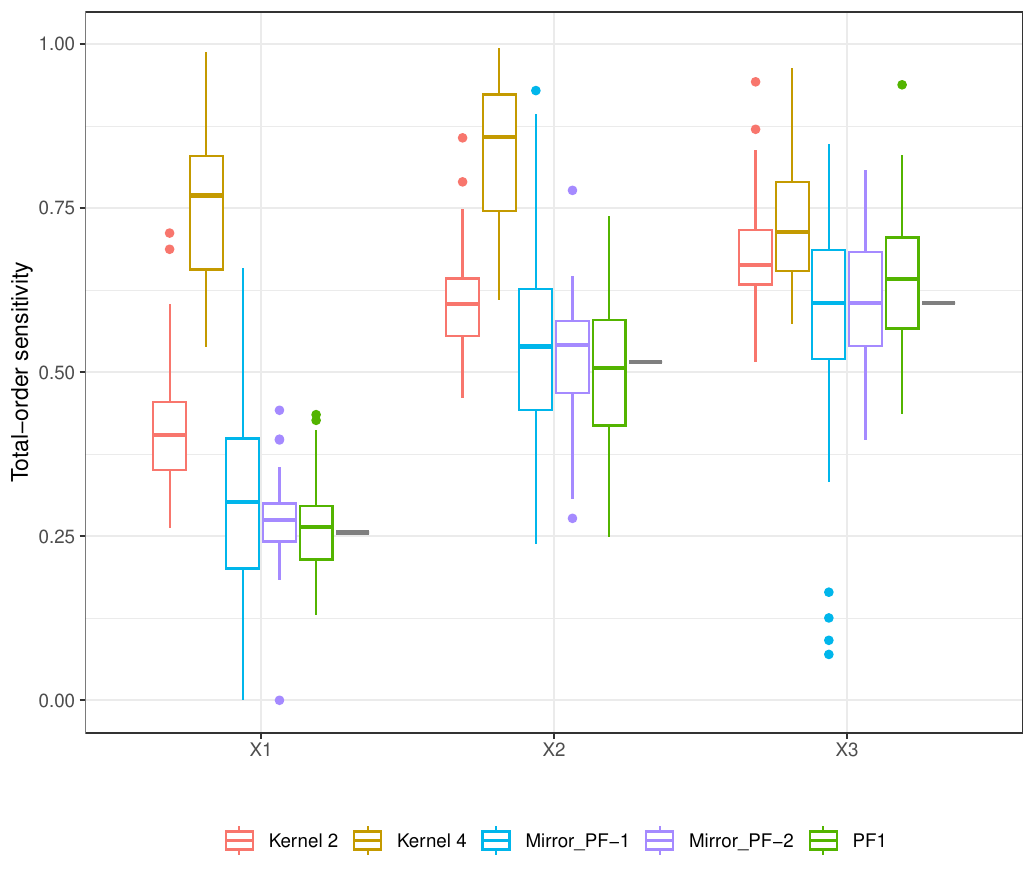}
    \caption{Estimators of the first-order indices (left panel) and the total order indices (right panel) for the function  $f_{disc}$ for the different methods. The reference value is represented with a gray line. }
    \label{fig:sobol_disc}
\end{figure}

\subsubsection{Comparison}

We first observe that, as expected, the "Mirror$\_$PF-1" method performs worse than "Mirror$\_$PF-2". This is entirely reasonable, since the comparison essentially reflects the behavior of our algorithm after 1 and $N_0$ epochs, respectively (where one epoch corresponds to the full set of $2^p$ variables in our problem). Nevertheless, in many cases, "Mirror$\_$PF-1", despite using an extremely small number of iterations, still achieves reasonably good performance given its very low computational cost for estimating \textit{all} Sobol' indices.

Next, it should be noted that the "Mirror$\_$PF-2" method consistently outperforms the baseline "PF1" method, which consists in independently repeating the Pick-Freeze procedure for each variable under consideration: in particular, the variance of "Mirror\_PF-2" are smaller than those of the Pick-Freeze method when considering the mirror algorithm at step $2^p N_0$. This clearly demonstrates the benefit of formulating our problem as a constrained optimization task, which allows all indices to be updated simultaneously.
Moreover, the "Mirror\_PF-1" method with a single epoch sometimes achieves results comparable to the baseline "PF1" (this is in particular the case for the null indices of $X^{(3)}, X^{(4)}$ and $X^{(5)}$ for $f_{\text{Bratley}}$), thus efficiently exploiting the simplex constraint in the search for the Sobol' indices.

Finally, the comparison of our methods with the kernel-based methods Kernel2 and Kernel4 requires more nuance. Indeed, for the regular Bratley function, the two kernel methods produce results that are quite close to each other, while, for an equivalent number of iterations in the algorithm, they perform similarly to the "Mirror$\_$PF-2" method and clearly outperform both "Mirror$\_$PF-1" and $PF1$. However, this observation for the class of regular functions completely disappears in the case of non-regular functions. In this case, we observe that the kernel estimators fail to estimate the correct Sobol' indices, whereas the Mirror Pick-Freeze and the Pick-Freeze methods still provide good results. 
The rather excellent results of kernel methods for smooth functions and the very disappointing performance for discontinuous functions are entirely consistent with the behavior of the kernel-based methods, which rely on a function decomposition that is not valid when the functions are non-smooth.
Finally, let us stress that a major drawback of kernel methods is the cost in memory required to run the estimation. On this aspect, our online algorithm performs at low memory cost even for a large number of input variables.

\subsection{Investigating three different sampling strategies}\label{sec:sampling_strat}
Finally, in this paragraph, we briefly investigate the influence of the sampling strategy. The distribution $a$ determines how our algorithm selects the different subsets throughout the iterations.
We choose to compare uniform sampling ("unif") with two adaptive strategies. In the latter, the distribution $a$ is updated at each step of the algorithm based on the current value of the sequence $\Sbn$. Specifically, we set $a \propto \Sbn$ in strategy "S" and $a \propto 1/\Sbn$ in strategy "1/S". We also compute an averaged version of the algorithm in the uniform case, referred to as "avg".
  
We perform 100 repetitions of the algorithm for different time horizons 
\[
n \in (2^p) \times \{500, 1000, 1500, 2500, 3500, 5000\}.
\]
We consider the discrete function $f_{disc,2}$ defined by
\[
f_{disc,2}(X^1, X^2, X^3) = X^1 \mathbf{1}_{\{X^3 > 0\}} + X^2 \mathbf{1}_{\{X^3 < 0\}} + X^3,
\]
and a step sequence $\eta_n = 6 / \sqrt{n+1}$.

The advantage of this choice lies in the fact that the true values of the Sobol' indices can be computed exactly. 
For each strategy, we compute the mean squared error of the estimated closed Sobol' indices, obtained as $M^{-1}\Sbn$.
For comparison, we also compute the $2^p$ Pick-Freeze estimators based on $\{500, 1000, 1500, 2500, 3500, 5000\}$ observations. The mean squared errors obtained in this case correspond to the line labeled "PF" in Figure \ref{fig:MSE}.
\begin{figure}[h!]
    \centering
    \includegraphics[scale=0.5]{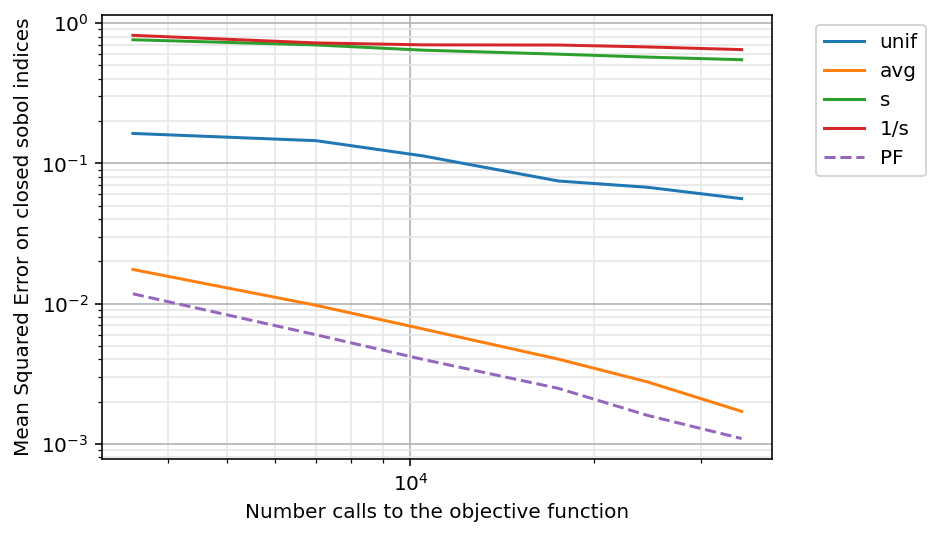}
    \caption{MSE on the $2^5$ closed Sobol' indices computed over 100 repetitions of the Pick-freeze mirror algorithm for different choices of the sampling distribution $a$.}
    \label{fig:MSE}
\end{figure}

Interestingly, we observe that the uniform strategy performs better than the adaptive strategies in this case. Furthermore, when considering an averaged version of the algorithm, we recover a convergence rate of $1/n$, similar to that of the Pick-Freeze method.
Finally, this last method, using a uniform weight for $a$ combined with Cesàro averaging, appears to achieve the best results among all our strategies.

These simulations illustrate that the choice of $a$ has a direct impact on the quality of our numerical results.
Similarly, the choice of $a$ also likely affects the performance of the averaged algorithm.
\section{Conclusions and perspectives}

{In this section, we gather comments on the results we have obtained, in particular their limitations and how they could be further extended or complemented.}

\paragraph{Discussion on the theoretical results (Theorems \ref{theo:cv_ps} and \ref{theo:sobol_non_asympt})}

{We note that the result stated in Theorem \ref{theo:cv_ps} is purely asymptotic, as it is expressed in terms of a limit in $n$, and it is not quantitative, in the sense that the “rate” of convergence of the sequence $(\Sbn)_{n \ge 1}$ toward the target $\Sb$ remains unknown, as our proof method does not allow such a rate to be identified.
In particular, it would be highly valuable to complement our result with an almost sure convergence rate (see \cite{AS-rate,pmlr-v134-sebbouh21a,JMLR:v25:23-1436}), which constitutes a stronger statement but would certainly require a more refined martingale decomposition.}

{Theorem \ref{theo:sobol_non_asympt} thus constitutes a first attempt to make our study of the proposed algorithm more quantitative.
We emphasize that item $i)$ of Theorem \ref{theo:sobol_non_asympt} is not  a true convergence result, which is indeed impossible to obtain for a constant step-size stochastic algorithm. Nevertheless, it can be regarded as a benchmark result, following common practice in convex optimization for machine learning. It provides a convenient way to assess the mean-square behavior of a stochastic optimization algorithm in a convex landscape.
In this context, obtaining an upper bound on the excess risk of order $1/\sqrt{n}$ is not surprising, as it corresponds to the minimax rate of convergence in many stochastic optimization problems with convex landscapes (see, e.g., \cite{NY83}).
Item $ii)$ of Theorem \ref{theo:sobol_non_asympt} concerns the case of a decreasing step sequence. The result holds for any finite simulation horizon but leads to a slightly weaker upper bound with an additional logarithmic factor in the bound of point $ii)$.}

{\paragraph{Towards a central limit theorem}
Even if the bound does not appear to explicitly involve the ambient dimension of the problem, it is implicitly present both in the term $\|a\|_{\exp,2}$ and in the fourth-order moment of the random variable $Y$.
This result preliminary captures the effect of the dimension $d$ of the inputs, and of the sample size $n$. Perhaps more concisely, it would be valuable to establish a central limit theorem for the properly rescaled algorithm $(\Sbn - \Sb)_{n \ge 1}$, which would provide a precise characterization of the limiting variance. 
This would allow us to obtain the exact rate of convergence, as well as to quantify the effect of the dimension through the limiting variance–covariance matrix.
See, \textit{e.g.}, \cite{duflo2013random,Gadat2016StochasticHB,10.1214/21-EJS1880} for examples of CLTs derived in various contexts using martingale representations or the Markov process perspective. Such extensions could be the subject of future research.}

{\paragraph{Choice of $a$}}

{Beyond understanding the dependence on $n$ and $d$, it would also be highly valuable to highlight the role of $a$, that is, the sampling distribution itself, in the results obtained by our method
as illustrated in Section \ref{sec:sampling_strat}. As a first step, by focusing primarily on the constant $\|a\|_{\exp,2}$. For a fixed number of coefficients, the quantity $\|a\|_{\exp,2}$ decreases as the distribution $a$ increasingly favors subsets of smaller size. Nevertheless, one should not conclude that it is advisable to choose a distribution $a$ such that $\|a\|_{\exp,2}$ is very small, since the result concerns the value of $\Phi^a$, not of $\bar{\Sbf}^{\eta}_{n}$ itself. Indeed, while it is possible to relate the proximity of $\Phi^a(\bar{\Sbf}^{\eta}_{n})$ to its minimum with the proximity of $\bar{\Sbf}^{\eta}_{n}$ to $\Sb$, this relation critically depends on the strong convexity constant of $\Phi^a$ (associated with $\nabla^2 \Phi^a$), which is significantly degraded if $a$ is chosen to produce a very small $\|a\|_{\exp,2}$.
 Hence, the influence of $a$ through $\|a\|_{\exp,2}$ and the spectral properties of $\nabla^2 \Phi^a$ warrants a deeper mathematical investigation, which lies beyond the scope of the present paper and will be addressed in future work.}
  
 {As a second step, one could also investigate the choice of $a$ by examining the limiting variance obtained through a potential central limit theorem. In particular, the study could focus on selecting $a$ by minimizing a functional involving the trace of the limiting covariance matrix (or any other criterion based on the resulting limiting distribution).}

{\paragraph{Numerical improvement}
Finally, one may adopt a more algorithmic perspective to extend our estimation strategy. It is relatively clear that any algorithm will suffer from the curse of dimensionality, since when $X$ contains $p$ variables, there are $2^p$ Sobol indices to estimate. Consequently, any method will exhibit an exponential dependence on $p$ for generic functions, in the absence of additional assumptions on the function $f$.
Faced with the combinatorial explosion of subsets, it is common in computer science to resort to greedy strategies of the divide-and-conquer type.
In order to design an efficient iterative method, one could consider adopting a hierarchical sequential algorithm that groups variables together through Sobol indices adapted to such groups. This constitutes a research topic in its own right, which deserves dedicated and sophisticated developments that we will address in future work.}

\appendix

\section{Proof of the theoretical results}
\label{sec:proofs}
This technical section is devoted to the proofs Theorem \ref{theo:cv_ps} and in Theorem \ref{theo:sobol_non_asympt}.

\subsection{Gradient of $\Phi^a$}

\begin{proof}[Proof of Proposition \ref{prop:gradient_Phi_a}]
For $s\in \Delta_q$ and $h\in\R^q$ we verify that:
\begin{align*}
\Phi^a(s+h)&-\Phi^a(s)\\& = \sum_{u \in \Ups} a_u \E[(Y-\E[Y])[M^{-1}h]_u \left((Y-\E[Y])[M^{-1}s]_u-(Y^u-\E[Y]) \right) ] +o(|h|^2)
\\
&=\sum_{u \in \Ups} a_u \E\left[(Y-\E[Y])\sum_{i} M^{-1}_{u,i}h_i \left((Y-\E[Y])[M^{-1}s]_u-(Y^u-\E[Y]) \right) \right] +o(|h|^2)
\\
&=\sum_{u \in \Ups} a_u \left\langle \E\left[(Y-\E[Y]) M^{-1}_{u,:} \left((Y-\E[Y])[M^{-1}s]_u-(Y^u-\E[Y]) \right) \right], h\right\rangle +o(|h|^2) \\
& = \langle \nabla \Phi^a(s),h \rangle + o(|h|^2)
\end{align*}
where $M^{-1}_{u,:}$ refers to row $u$ of the matrix $M^{-1}$. 
Let us finally remark that if $U$ is a discrete $\Ups$-valued random variable distributed according to the probability distribution $a$, then we can rewrite $\nabla \Phi^a$ as:
\begin{align*}
\nabla \Phi^{a}(s)&=\sum_{u\in \Ups} a_u  \E\left[(Y-\E[Y]) M^{-1}_{u,:} \left((Y-\E[Y])[M^{-1}s]_u-(Y^u-\E[Y]) \right) \right]    \\
&=\E\left[(Y-\E[Y]) M^{-1}_{U,:} \left((Y-\E[Y])[M^{-1}s]_U-(Y^U-\E[Y]) \right) \right].
\end{align*}
\end{proof}

\subsection{Proof of Theorem \ref{theo:cv_ps} (almost sure convergence)}
The proof of almost sure convergence follows the general approach of \cite{Lan:2012}, with specific attention given to the presence of an estimation bias in the algorithm due to the online approximation of $\E[Y]$.
This bias is handled in a manner quite similar to that of \cite{CGH24}, although the control of this bias term is, of course, specific to the context considered in the present work.
Finally, although the approach based on relative smoothness introduced in \cite{bauschke2017descent} offers a powerful framework for analyzing mirror descent methods (even in the stochastic cases, see \cite{Dragomir,Dragomir-Taylor}), it cannot be employed in our setting, as the conditions required to apply the results of  \cite{Dragomir,Dragomir-Taylor}
are not satisfied. Indeed, while the function $\Phi^a$ exhibits uniform Hessian bounds over the entire simplex (see in particular Appendix \ref{app:spectre-M}), the entropy function $h$ given in Equation \eqref{def:phi} defining our Bregman divergence is not differentiable on the boundary of the simplex so that the metric induced by $\Dp$ cannot be upper-bounded by the one induced by the Hessian of $\Phi^a$.

\begin{proof}[Proof of Theorem \ref{theo:cv_ps}]
 $ $\\
The proof is organized into three parts. The first part leverages the variational formulation of mirror descent to establish a one-step inequality for the Bregman divergences. The second part provides a precise, quantitative estimate of the bias between the estimated drift at iteration $n$ and the theoretical drift, which arises from replacing $\E[Y]$ with the empirical mean $\hat{m}_n$ in the gradient. In the final part, the convergence of the sequence is established using the Robbins–Siegmund lemma, together with several technical estimates.\\

\noindent \textbf{Step 1: Obtaining a recursion on\, $\Dp(\Sb,\Sbn)$} 
Recall the definition of $\Sbnp$
\[\Sbnp = \arg \min_{s \in \Delta_q } \left\{ \langle 
\Hhnp(\Sbn),s -\Sbn \rangle + \frac{1}{\eta_{n+1}} \Dp(s,\Sbn)\right\}
\,,\]
then the first order condition reads
\[
\forall s\in\Delta_q: \qquad 
\eta_{n+1} \langle\Hhnp(\Sbn) , s-\Sbnp \rangle + \langle \nabla_1 \Dp (\Sbnp,\Sbn), s-\Sbnp\rangle \ge0,
\]
where $\nabla_1 \Dp$ refers to the gradient with respect to the first variable.
As  $\nabla_1 \Dp(x,y)=\nabla h(x) -\nabla h(y)$ (Lemma \ref{lem:grad_dp}) we deduce that:
\begin{align*}
\forall s\in\Delta_q: \qquad 
\eta_{n+1} \langle\Hhnp(\Sbn), \Sbnp-s \rangle 
&\le \langle \nabla h(\Sbnp)-\nabla h(\Sbn), s-\Sbnp\rangle\\
& \le \Dp (s,\Sbn)-\Dp(s,\Sbnp)-\Dp (\Sbnp,\Sbn),
\end{align*}
where the second inequality derives from the three point lemma \ref{lem:3points}. Therefore, using simple algebra and Inequality \eqref{eq:Dp_lower}, we obtain that:
\begin{align*}
\Dp (s,\Sbnp)\le  \Dp (s,\Sbn)-\eta_{n+1} \langle\Hhnp(\Sbn), \Sbnp -s\rangle -\frac{1}{2}\| \Sbnp-\Sbn\|_2^2.
\end{align*}
We can now substitute $\Sbnp -s$ by $\Sbn -s$ in the middle term and write that:
\begin{align*}
\eta_{n+1} \langle\Hhnp(\Sbn), \Sbnp -s\rangle=\eta_{n+1} \langle\Hhnp(\Sbn), \Sbn -s\rangle +\eta_{n+1} \langle\Hhnp(\Sbn), \Sbnp -\Sbn\rangle.
\end{align*}
The Young inequality induces:
\[\eta_{n+1}|\langle\Hhnp(\Sbn), \Sbnp -\Sbn\rangle|\le \frac{1}{2} \eta_{n+1}^2 \|\Hhnp(\Sbn)\|_2^2 +\frac{1}{2} \| \Sbnp-\Sbn\|_2^2.
\]
This finally leads to the key inequality:
\begin{equation}\label{eq:rec_1}
\Dp (s,\Sbnp)\le  \Dp (s,\Sbn)-\eta_{n+1} \langle\Hhnp(\Sbn), \Sbn -s\rangle +\frac{1}{2} \eta_{n+1}^2 \|\Hhnp(\Sbn)\|_2^2.
\end{equation}

\noindent\textbf{Step 2: Comparison between   $\Hhnp(\Sbn)$ and  $ \nabla \Phi^a(\Sbn)$.}
Recall that from Proposition \ref{prop:gradient_estimator_unbiased}
\[\Phi^a(\Sbn)= \E\left[ \nabla H(\Sbn, \Ynp,\YUnp,\Unp)| \mathcal{F}_n \right],
\] 
we can then write 
\begin{equation}
    \label{eq:dec-drift}
    \Hhnp(\Sbn)=\nabla \Phi^a(\Sbn) + \Delta M_{n+1} +R_{n+1},
\end{equation}
where $\Delta M_{n+1}$ is a $(\mathcal{F}_n)$-martingale increment defined by:
\begin{equation}
\label{def:martingale}
\Delta M_{n+1}=\Hhnp(\Sbn)-\E\left[\Hhnp(\Sbn)| \mathcal{F}_n \right],
\end{equation}
and $R_{n+1}$ is a rest/bias term that comes  from the estimation of $\mathbb{E}[Y]$ with $\hat{m}_n$ in Equation \eqref{eq:gradient_estimator}:    
\begin{equation}
    \label{def:reste}
R_{n+1} = \E\left[\Hhnp(\Sbn)| \mathcal{F}_n \right] - \nabla \Phi^a(\Sbn).
\end{equation}

Using Equation \eqref{eq:dec-drift} in \eqref{eq:rec_1}, we deduce that:
\begin{align*}
\Dp (s,\Sbnp)&\le  \Dp (s,\Sbn)-\eta_{n+1} \langle\nabla \Phi^a(\Sbn) , \Sbn -s\rangle-\eta_{n+1} \langle\Delta M_{n+1} , \Sbn -s\rangle \\
&\qquad -\eta_{n+1} \langle R_{n+1}, \Sbn -s\rangle +\frac{1}{2} \eta_{n+1}^2 \|\Hhnp(\Sbn)\|_2^2.
\end{align*}
Let us choose $s=\Sb$ in the previous inequality to obtain 
\begin{align}
\Dp (\Sb,\Sbnp)&\le  \Dp (\Sb,\Sbn)-\eta_{n+1} \langle\nabla \Phi^a(\Sbn) , \Sbn -\Sb\rangle-\eta_{n+1} \langle\Delta M_{n+1} , \Sbn-\Sb\rangle\nonumber\\
&\qquad-\eta_{n+1} \langle R_{n+1}, \Sbn -\Sb\rangle +\frac{1}{2} \eta_{n+1}^2 \|\Hhnp(\Sbn)\|_2^2. \label{eq:ineg_1}
\end{align}
Now from the Cauchy-Schwarz inequality and then the Young inequality, we have
\begin{align}
-\eta_{n+1} \langle R_{n+1}, \Sbn -\Sb\rangle \le 
& \  \eta_{n+1}\|R_{n+1}\|_2 \|\Sbn -\Sb\|_2 \nonumber\\
& \le  \eta_{n+1}\|R_{n+1}\|_2 \frac{1+\|\Sbn -\Sb\|_2^2}{2} \,. \label{eq:upper_bound_reste}    
\end{align}
We now use Inequality \eqref{eq:upper_bound_reste} in \eqref{eq:ineg_1} together with the strong convexity of $h$ in inequality \eqref{eq:Dp_lower} to obtain:
\begin{align*}
\Dp (\Sb,\Sbnp)&\le  \Dp (\Sb,\Sbn)(1+{\frac{1}{2}}\eta_{n+1}\| R_{n+1}\|_2)
 -\eta_{n+1} \langle\nabla \Phi^a(\Sbn) , \Sbn -\Sb\rangle \\
 &-\eta_{n+1} \langle\Delta M_{n+1} , \Sbn-\Sb\rangle +\frac{1}{2}\eta_{n+1}\| R_{n+1}\|_2+\frac{1}{2} \eta_{n+1}^2 \| \Hhnp(\Sbn)\|_2^2.
\end{align*}
We observe that $R_{n+1}$ is $\mathcal{F}_n$-measurable and computing the conditional expectation with respect to $\mathcal{F}_n$ leads to: 
\begin{align}
\E[\Dp (\Sb,\Sbnp)|\mathcal{F}_n] &\le  \Dp (\Sb,\Sbn) (1+ {\frac{1}{2}}\eta_{n+1} \|R_{n+1}\|_2)
-\eta_{n+1} \langle\nabla \Phi^a(\Sbn) , \Sbn -\Sb\rangle
\nonumber
\\
&\qquad+ \frac{1}{2}\eta_{n+1}\| R_{n+1}\|_2+\frac{1}{2} \eta_{n+1}^2 \E[\| \Hhnp(\Sbn)\|_2^2|\mathcal{F}_n].
\label{ineq:clef}
\end{align}
\noindent\textbf{Step 3: Conclusion of the proof.}
Using Lemma \ref{lem:R_n} and our set of assumptions on the step-size sequence \eqref{hyp:pas}, we deduce that:
\[\sum_{n \ge 0} \eta_{n+1} \E\left[\|R_n\|_2\right]<\infty.\]
Notice that since all the terms are non-negative, 
\[\E\left[\sum_{n \ge 0}  \eta_{n+1} \|R_n\|_2\right]\le \sum_{n \ge 0}\eta_{n+1} \E[\|R_n\|_2],\] and it shows that $\sum_{n \ge 0} \eta_{n+1} \|R_n\|_2$ is almost surely finite and $\prod_{n \ge 0} (1+ \eta_{n+1} \|R_n\|_2)<\infty$ almost surely as well.
In the meantime, Lemma \ref{lem:norm} associated with conditions \eqref{hyp:pas} also shows that:
\[ \sum_{n \ge 0}\E[\eta_{n+1}^2 \| \Hhnp(\Sbn)\|_2^2 ] <\infty.
\]
We can now apply the Robbins-Siegmund Lemma (stated in Lemma \ref{lem:Robins-Siegmund}) and deduce that a random variable $D_\infty$ exists such that: 
\begin{equation} \label{eq:conv-D-infty}
\Dp (\Sb,\Sbnp) \longrightarrow D_\infty  \in L^1 \quad \text{and}\quad \sum_{n \ge 0} \eta_{n+1} \langle\nabla \Phi^a(\Sbn) , \Sbn -\Sb\rangle< \infty,\quad  \text{a.s.}     
\end{equation}
Using the convexity of $\Phi^a$ we obtain
\begin{equation}\label{eq:conv-series} \sum_{n \ge 0} \eta_{n+1} \left(\Phi^a(\Sbn)\right) \le  \sum_{n \ge 0} \eta_{n+1} \langle\nabla \Phi^a(\Sbn) , \Sbn -\Sb\rangle <+\infty, \,\quad \text{a.s.}  
\end{equation}
The rest of the proof proceeds now in a standard way in stochastic approximation theory. From Equation \eqref{eq:conv-series} and $\sum_{n \ge 0} \eta_{n+1} = + \infty$, we know that a subsequence $(\hat{\mathbf{S}}_{n_k})_{k \ge 1}$ exists such that:
$$
\lim_{k \rightarrow + \infty} \langle\nabla \Phi^a(\hat{\mathbf{S}}_{n_k}) , \hat{\mathbf{S}}_{n_k} -\Sb\rangle =0, \quad  \text{a.s.}  
$$
The strong convexity of $\Phi^a$ may be translated into the following inequality:
$$
\forall s \in \Delta_q \qquad 
\langle\nabla \Phi^a(s) , s -\Sb\rangle \ge \rho_a \|s-\Sb\|^2,
$$
where $\rho_a>0$ refers to the lowest eigenvalue of $\nabla^2 \Phi^a$ over $\Delta_q$ (see Proposition \ref{prop:hessienne} in the appendix). 

It then implies that 
$$
\lim_{k \rightarrow + \infty} \hat{\mathbf{S}}_{n_k} = \Sb, \quad 
\text{a.s.}  
$$
In the meantime, we also deduce from Equation \eqref{eq:conv-D-infty} that $\Dp (\Sb,\hat{\mathbf{S}}_{n_k}) \longrightarrow D_\infty$ a.s. and the continuity of the Bregman divergence $\Dp$ yields $D_\infty=0$ a.s.
Finally, the lower bound \eqref{eq:Dp_lower} implies that the entire sequence $\Sbn$ converges towards $\Sb$.

\end{proof}

\subsection{Proof of Theorem \ref{theo:sobol_non_asympt} ( non-asymptotic upper bound)}

The following proof follows an approach recently proposed in \cite{CGH24} to obtain a non-asymptotic bound for biased stochastic mirror methods. Nevertheless, care must be taken to meticulously track the sequence of inequalities in order to preserve the dimension-dependent constants that have a significant influence on the final bound. Moreover, although the proof strategy is similar to that of \cite{CGH24}, it still requires bias controls that are specific to the model considered here, as well as precise inequalities on the spectra of the change-of-variable matrices $M$ used.

\begin{proof}[Proof of Theorem \ref{theo:sobol_non_asympt}]
The starting point of the non-asymptotic bound is Equation \eqref{ineq:clef} that can be written as
\begin{align*}
    \E [\Dp(\Sb,\Sbnp)|\mathcal{F}_n]& \leq \Dp(\Sb,\Sbnp) - \eta_{n+1} \langle \nabla \Phi^a(\Sbn),\Sbn-\Sb\rangle \\
    &\quad + \frac{\eta_{n+1} \|R_{n+1}\|_2}{2} (1+\|\Sbn-\Sb\|_{2}^2) + \frac{1}{2} \eta_{n+1}^2 \E[\| \Hhnp(\Sbn)\|_2^2|\mathcal{F}_n]
    \\
    & \leq \Dp(\Sb,\Sbnp) - \eta_{n+1}  (\Phi^a(\Sbn)-\Phi^a(\Sb))\\
    &\quad +\frac{5}{2}\eta_{n+1} \|R_{n+1}\|_2 + \frac{1}{2} \eta_{n+1}^2 \E[\| \Hhnp(\Sbn)\|_2^2|\mathcal{F}_n],
\end{align*}
where we used in the last line the convexity of $\Phi^a$ that yields $-\langle \nabla \Phi^a(\Sbn),\Sbn-\Sb \rangle \leq -(\Phi^a(\Sbn)-\Phi^a(\Sb))$ and the compactness of $\Delta_q$ that yields $\|\Sbn-\Sb\|_2^2 \leq \|\Sbn-\Sb\|_1^{2} \leq 4$. 
We now compute the overall expectation, use Lemma \ref{lem:R_n} and Lemma \ref{lem:norm}, and obtain that:
\begin{align*}
\eta_{n+1} \E[\Phi^a(\Sbn)-\Phi^a(\Sb)] &\leq  \E [ \Dp(\Sb,\Sbn) - \Dp(\Sb,\Sbnp)]\\
&\quad+5 \eta_{n+1} \E [ \|R_{n+1}\|_2 ] + 
\frac{\eta_{n+1}^2}{2} \E[\| \Hhnp(\Sbn)\|_2^2]\\
&\leq  \E [ \Dp(\Sb,\Sbn) - \Dp(\Sb,\Sbnp)]\\
& \quad +5  \|a\|_{\exp} \Var(Y) \frac{\eta_{n+1} }{n} + C \left(1+\sqrt{\|a\|_{\exp,2}} \right) \E[Y^4]\frac{\eta_{n+1}^2}{2}.
\end{align*}
The rest of the proof then proceeds following a standard argument with Cesaro mean and telescopic sums:
\begin{align*}
\sum_{k=1}^n \eta_{k+1} (\E[\Phi^a(\Sbk)] - \Phi^a(\Sb)) 
&\leq \Dp(\Sb,\Sbun) + 5 \|a\|_{\exp} \Var(Y) \sum_{k=1}^n \frac{\eta_{k+1} }{k}\\
&\quad+  C \left(1+\sqrt{\|a\|_{\exp,2}} \right) \E[Y^4] \sum_{k=1}^n \frac{\eta_{k+1}^2}{2}.
\end{align*}
Defining now the weighted Cesaro average as:
$$
\bar{\Sbf}^{\eta}_{n} = \frac{\sum_{k=1}^n \eta_{k+1} \Sbk}{\sum_{k=1}^n \eta_{k+1}},
$$
then, the convexity of $\Phi^a$ yields:
\begin{align}
    \E[\Phi^a(\bar{\Sbf}^{\eta}_{n})]&-\min \Phi^a  
    \leq \frac{
\sum_{k=1}^n \eta_{k+1} (\E[\Phi^a(\Sbk)] - \Phi^a(\Sb))}{\sum_{k=1}^n \eta_{k+1}} \nonumber\\
& \leq \frac{ \Dp(\Sb,\Sbun) + 5\|a\|_{\exp} \Var(Y)
\sum_{k=1}^n  \eta_{k+1} k^{-1} + \frac{ C}{2} \left(1+\sqrt{\|a\|_{\exp,2}} \right) \E[Y^4]}{  \sum_{k=1}^n \eta_{k+1}}.
  \label{eq:borne-S_eta_n}
\end{align}
It remains to optimize both terms of the right hand side of  Inequality \eqref{eq:borne-S_eta_n} by choosing an appropriate step-size sequence $(\eta_{k+1})_{k\ge0}$. There exists two classical choices.
\begin{itemize}
\item Either $(\eta_{k+1})_{k\ge0}$ is piecewise constant (depending on the final number of performed simulations):
\[ \eta_{k+1}=\begin{cases}& \eta^*\quad \forall k  \le n-1\\
&0 \quad \forall k>n-1.\end{cases}
\] 
In that case, Inequality \eqref{eq:borne-S_eta_n} reads:
\begin{align*} \E[\Phi^a(\bar{\Sbf}^{\eta}_{n})]-\min \Phi^a 
&\le \frac{\Dp(\Sb,\Sbun)}{n\eta^*}+  \frac{C}{2} \left(1+\sqrt{\|a\|_{\exp,2}} \right) \E[Y^4] \eta^* \\
&\qquad + 5\|a\|_{\exp} \Var(Y)
\frac{\log(n)}{n} .
\end{align*}
We finally optimze the choice of $\eta^*$, which leads to:
$$\eta^*=\sqrt{\frac{1}{ (1+\sqrt{\|a\|_{\exp,2}}) n}}.$$
Our final upper bound is then:
 \begin{align*}
 \E[\Phi^a(\bar{\Sbf}^{\eta}_{n})]-\min \Phi^a &\leq \sqrt{ 1+\sqrt{\|a\|_{\exp,2}}} \left( \Dp(\Sb,\Sbun) 
 + \frac{C}{2} \E[Y^4] \right) n^{-1/2}\\&\qquad + 5\|a\|_{\exp} \Var(Y) \frac{\log(n)}{n}.
 \end{align*}
\item Or $(\eta_{k+1})_{k\ge0}$ is a decreasing step-size sequence of the form $\eta_{k+1}=\eta_0 (k+1)^{-\alpha}$ with $\alpha\in[1/2,1)$.
In that case the right hand side of Inequality \eqref{eq:borne-S_eta_n} decays as $(\sum_{k=1}^n  \eta_{k+1})^{-1} \le \eta_0 (n+1)^{-\alpha+1}$.
The situation becomes somewhat more technical and tedious to describe precisely, and only the dependence on the number of iterations remains straightforward to characterize. In particular, one can verify that with $\alpha=1/2$ and $\eta_0 \asymp 1$, a large enough $C$ exsists such that:
$$
\E[\Phi^a(\bar{\Sbf}^{\eta}_{n})]-\min \Phi^a
\leq C \left( \Dp(\Sb,\Sbun)  + \|a\|_{\exp,2} +  \E[Y^4] \right) \frac{\log n}{\sqrt{n}},
$$
which constitutes a slightly weaker bound than the inequality obtained in the previous case.
\end{itemize}

 \end{proof}
 \subsection{Controls on the bias terms}

\begin{lemma}\textbf{Control of $\E[\|R_{n+1}\|_2]$}\label{lem:R_n}
For any $n \ge 1$, one has:  \[ \E[\|R_{n+1}\|_2]\le  \frac{ \|a\|_{\exp,1/2}  \Var[Y] }{n}.\] 

\end{lemma}

\begin{proof}[Proof of Lemma \ref{lem:R_n}]
We recall that the bias term involved in $R_{n+1}$ is given by:
\[R_{n+1} = \E\left[ \Hhnp(\Sbn)| \mathcal{F}_n \right] - \nabla \Phi^a(\Sbncorr).
\]
According to definition \eqref{def:reste}, we are led to compute:
\begin{align*}
    \E&\left[ \Hhnp(\Sbn)|\mathcal{F}_n\right]\\
    &=\E\left[ (\Ynp-\mn) \left( (\Ynp-\mn)[\Mbm s]_{\Unp}-(\YUnp-\mn) \right)
[\Mbm]_{\Unp,:} ^T |\mathcal{F}_n\right].
\end{align*}
Let us use the decompositions: 
$$\Ynp-\mn= (\Ynp-\E[Y])+(\E[Y]-\mn) $$ and $$ \YUnp-\mn = (\YUnp-\E[Y]) + (\E[Y]-\mn)
$$
and the simple algebra:
\[ (\alpha+\epsilon)\Big((\alpha+\epsilon)C-(\beta+\epsilon) \Big) = \alpha\Big(\alpha C-\beta-\alpha\Big) +\epsilon (2\alpha C-\alpha-\beta ) +\epsilon^2(C-1),\]
we deduce that
\begin{align*}
     &\Hhnp(s) =\nabla H(s,\Ynp,\YUnp,\Unp)\\
     & + (\E[Y]-\mn)\left( 2(\Yn-\E[Y])[\Mbm s]_{\Unp} -(\Yn-\E[Y]) - (\YUnp -\E[Y]) \right)[\Mbm]_{\Unp,:} ^T\\
&+ (\E[Y]-\mn)^2 \left( [\Mbm s]_{\Unp}-1\right)[\Mbm]_{\Unp,:} ^T .
\end{align*}
We then compute the conditional expectation with respect to $\mathcal{F}_n$ of the previous expression and combine our last decomposition with Equation \eqref{eq:gradient_estimator} to get:
\begin{align}
R_{n+1}=&(\E[Y]-\mn)\E\left[\left\{ 2(\Yn-\E[Y])[\Mbm \Sbncorr]_{\Unp}-(\Yn-\YUnp)  \right\}  [\Mbm ]_{\Unp,:} ^T|\mathcal{F}_n\right] \nonumber\\
&+(\E[Y]-\mn)^2\E[([\Mbm \Sbncorr]_{\Unp}-1)[\Mbm]_{\Unp,:} ^T|\mathcal{F}_{k}].
\label{eq:Rk_expression}
\end{align}
We first consider the conditional expectation:
\begin{align*}
   & \E\left[\left\{ 2(\Yn-\E[Y])[\Mbm \Sbn]_{\Unp}-(\Yn-\YUnp)  \right\}  [\Mbm ]_{\Unp,:} ^T|\mathcal{F}_n\right] \\
    &\quad=    \E\left[ 2(\Yn-\E[Y])[\Mbm \Sbn]_{\Unp} [\Mbm ]_{\Unp,:} ^T|\mathcal{F}_n\right] \\
    &\qquad-\E\left[(\Yn-\YUnp)   [\Mbm]_{\Unp,:} ^T|\mathcal{F}_n\right].
\end{align*}
For the first term remark that $\Yn$ is independent of $\Unp$ and of $\mathcal{F}_n$, therefore
\begin{align*}
 &    \E\left[ 2(\Yn-\E[Y])[M^{-1}\Sbncorr]_{\Unp} [\Mbm ]_{\Unp,:} ^T|\mathcal{F}_n\right]\\
&=    2 E[ \Yn-\E[Y]] \E\left[[M^{-1}\Sbncorr]_{\Unp} [\Mbm ]_{\Unp,:} ^T |\mathcal{F}_n\right]
= 0.
\end{align*}
Similarly, we can rewrite the second term as 
\begin{align*}
   & E\left[(\Yn-\YUnp)   [\Mbm ]_{\Unp,:} ^T|\mathcal{F}_n\right] \\
& = E\left[((\Yn-\E[Y]) + (E[Y]- \YUnp)  ) [\Mbm ]_{\Unp,:} ^T|\mathcal{F}_n\right] \\
 &= -E\left[(\YUnp-\E[Y])   [\Mbm ]_{\Unp,:} ^T|\mathcal{F}_n\right] \\
 &= -E\left[(Y^{(U)}-\E[Y])   [\Mbm ]_{U,:} ^T\right].
\end{align*}
Let us consider a coordinate of the above column vector
\begin{align*}
\E[ (Y^{(U)}-\E[Y]) \Mbm _{U,j}] = \E \left[ \E[  (Y^{(U)}-\E[Y])  |U] \Mbm _{U,j}\right] =0.
\end{align*}
As a consequence, the bias term simply reduces to the following $q$-dimensional vector:
\begin{align}
R_{n+1}=(\E[Y]-\mn)^2\E[([\Mbm \Sbncorr]_{\Unp}-1)[\Mbm ]_{\Unp,:} ^T|\mathcal{F}_{n}].
\label{eq:Rk_expression_2}
\end{align}
We are led  to compute for $s\in\Delta_q$, and $U$ the discrete random variable distributed according to $a$ the following vector:
$\E[([\Mbm s]_{U}-1)[\Mbm]_{U,:} ^T].$
At this stage, we use the value of $\Mbm$ given in Equation \eqref{def:M_inv}, the triangle inequality and get:
$$
\|R_{n+1}\|_2 \leq (\E[Y]-\mn)^2 \E\left[ |[\Mbm \Sbncorr]_{\Unp}-1|  \left\| \Mbm _{\Unp,:} \right\|_2 |\mathcal{F}_{n}\right].
$$
Since for any $s \in \Delta_q$, we have $[\Mbm s]_U = \sum_{v \subset u} s_v \in [0,1]$ and that 
\[\|\Mbm_U\|_2^2 = \sum_{V \in \Ups} \mathbf{1}_{V \subset U} = 2^{|U|},\]
we then use our definition of $ \|a\|_{\exp,1/2}$ given in Equation \eqref{eq:norme-a-power} and get:
$$
\|R_{n+1}\|_2 \leq  \|a\|_{\exp,1/2} (\E[Y]-\mn)^2.
$$
We finally obtain the desired result while considering the global expectation:
\[\E[\|R_n\|_2]\le \|a\|_{\exp,1/2}  \E[(\E[Y]-\mn )^2] \le \|a\|_{\exp,1/2}  \frac{\Var[Y]}{n}.\]

\end{proof}

\begin{lemma}\textbf{Control of $\E[\|\Hhnp(\Sbn)\|^2]$}
\label{lem:norm}
A constant $C$ independent of $p$ exists such that:
\[\E\left[\|\Hhnp(\Sbn)\|^2\right] \le C \left(1+\sqrt{\|a\|_{\exp,2}}\right) \E[Y^4]. \]
\end{lemma}

\begin{proof}[Proof of Lemma \ref{lem:norm}]
Let us recall that 
\[ \Hhnp(\Sbn)=(\Ynp-\mn) \left( (\Ynp-\mn)[\Mbm \Sbncorr]_{\Unp}-(\YUnp-\mn) \right)
[\Mbm]_{\Unp,:}.
\]
We can then compute the square norm and obtain:
\begin{align*}
&\|\Hhnp(\Sbn)\|^2 = (\Ynp-\mn)^2 \left( (\Ynp-\mn)[\Mbm \Sbncorr]_{\Unp}-(\YUnp-\mn) \right)^2
\|[\Mbm]_{\Unp,:}\|^2\\
&\le \left\{2 (\Ynp-\mn)^4 ([\Mbm \Sbncorr]_{\Unp})^2 +2 (\Ynp-\mn)^2(\YUnp-\mn)^2\right\} \|[\Mbm]_{\Unp,:}\|^2.
\end{align*}
We apply the Cauchy-Schwarz inequality and obtain:
\begin{align*}
    \E\left[\|\Hhnp(\Sbn)\|^2 \right] & \leq 
    2 \E \left[ (\Ynp-\mn)^4 ([\Mbm \Sbncorr]_{\Unp})^2\right]\\
    & \quad + 2
    \sqrt{ \E \left[ (\Ynp-\mn)^4 \|[\Mbm]_{\Unp,:}\|^4 \right]}
    \sqrt{\E \left[(\YUnp-\mn)^4 \right]}\\
    & \leq  2 \E \left[ (\Ynp-\mn)^4 ([\Mbm \Sbncorr]_{\Unp})^2\right]\\
   & \quad + 2
    \sqrt{ \E \left[ (\Ynp-\mn)^4\right]} \sqrt{ \E \left[\|[\Mbm]_{\Unp,:}\|^4 \right]}
    \sqrt{\E \left[(\YUnp-\mn)^4 \right]},
\end{align*}
where in the last line we used the independence between $\Ynp$ and $\Unp$. From the same argument as the one used in Lemma \ref{lem:R_n}, we have for any $s\in\Delta_q$:
\begin{equation} [\Mbm s]_{U} = \sum_{v \subset U} s_v \in [0,1] \quad \text{and} \quad \|[\Mbm]_{U,:}\|_2^4 = {\left( \sum_{v \in \Ups} \mathbf{1}_{v \subset U} \right)^4}= 2^{2|U|}.
\label{eq:norme-phi}
\end{equation}
We then deduce that:
\begin{align*}
    \E\left[\|\Hhnp(\Sbn)\|^2 \right] & \leq   2 \E \left[ (\Ynp-\mn)^4 \right]+
    \E \left[ (\Ynp-\mn)^4\right] \sqrt{ \E \left[\|[\Mbm]_{\Unp,:}\|^4 \right]}\\
    & \leq 16 \left( \E[(\Ynp-\E[Y])^4] + \E[(\E[Y]-\mn)^4]\right) \left(1+\sqrt{ \E \left[\|[\Mbm]_{\Unp,:}\|^4 \right]}\right)\\
    & \leq 16 \left( \E[(\Ynp-\E[Y])^4] + C_4 \frac{\E[Y^4]}{n^2}\right) \left(1+\sqrt{\|a\|_{\exp,2}}\right),
\end{align*}
where the last bound follows from Lemma \ref{lem:Petrov}.
Combining all these inequalities proves the result.
\end{proof}

\section{Standard properties on (stochastic) optimization methods\label{app:tec}}

We list below some technical results that are   useful for our purpose in the proof of Theorems \ref{theo:cv_ps} and \ref{theo:sobol_non_asympt}.

\paragraph{Properties of Bregman divergence}
We recall some standard results that are valid for any Bregman divergence $\Dp$. We refer to \cite{nemirovsky1983wiley} for further details.
\begin{lemma}[Three points lemma]\label{lem:3points}
For any triple of points $(x,y,z)$, one has:
$$
\Dp(x,z)=\Dp(x,y)+\Dp(y,z)-\langle \nabla h(z)-\nabla h(y),x-y\rangle.
$$
\end{lemma}
\begin{lemma}[Gradient of the Bregman divergence]\label{lem:grad_dp}
For any pair of points $(x,y)$, one has:
$$
\nabla_x \Dp(x,y) = \nabla h(x)-\nabla h(y).
$$
\end{lemma}

\paragraph{Moments of sums of random variables}
\begin{prop}\label{lem:Petrov}[Dharmadhikari and Jogdeo, see \cite{petrov1975sums}, Example 16 p. 60]
 Assume that $X_1,\ldots,X_n$ are centered independent random variables with finite moments of order $p \ge 2$. Then an explicit constant $C_p$ exists such that:
 $$
 \mathbb{E}\left[ |X_1+\ldots+X_n|^p\right] \leq C_p n^{p/2-1} \sum_{k=1}^n \E |X_k|^p,
 $$
 Moreover, $C_p$ is  universal and given by:
 $$
 C_p = \frac{p(p-1)}{2} (1 \vee 2^{p-3}) \left(1+\frac{2}{p} K_{2m}^{(p-2)/2m}\right) \quad \text{with} \quad m = \lfloor p/2 \rfloor \quad \text{and} \quad K_{2m} = \sum_{r=1}^{2m} \frac{r^{2m-1}}{(r-1)!)}.
 $$
\end{prop}
\paragraph{Robbins Siegmund Lemma}
We shall state one of the most useful result on stochastic algorithms. This result is
known as the Robbins-Siegmund Lemma and stated below.
\begin{lemma}[Robbins-Siegmund Lemma, see \cite{robbins-siegmund}]\label{lem:Robins-Siegmund}
Consider a filtration $(\mathcal{F}_n)_{n \ge 0}$ and 4 sequences of random variables $(A_n)_{n \ge 0}$, $(B_n)_{n \ge 0}$, $(\alpha_n)_{n \ge 0}$ and $(\beta_n)_{n \ge 0}$ that are $\mathcal{F}_n$-measurables, non-negatives and integrables such that these sequences enjoy the next assumptions:
\begin{itemize}
    \item[(1)] $(\alpha_n)_{n \ge 0}$, $(A_n)_{n \ge 0}$ and $(\beta_n)_{n \ge 0}$ are $\mathcal{F}_n$-predictables.
    \item[(2)] $\sup_{\omega \in \Omega} \prod_{n \ge 1} (1+\alpha_n(\omega)) < +  \infty $ and $\sum_{n \ge 0} \E[\beta_n] < + \infty$.
    \item[(3)] $\E[B_{n+1} \vert \mathcal{F}_n] \leq (1+\alpha_{n+1}) B_n +\beta_{n+1} - A_{n+1}. $
\end{itemize}
 Then, 
 \begin{itemize}
    \item[(i)]
    $B_n \longrightarrow B_{\infty}$ in $L^1$ where $B_{\infty}$ is an integrable random variable and $\sup_{n \ge 1} \E[B_n]<+\infty$.
    \item[(ii)] $\sum_{n \ge 0} \E[A_n] < \infty$ and $\sum_{n \ge 0} A_n< + \infty$ almost surely.
\end{itemize}
\end{lemma}

\section{Properties on the matrix $\Mb$\label{app:spectre-M} and on the spectrum of $\nabla^2 \Phi^a$}

This section is devoted to obtaining properties of the matrix $\Mb$ depending on the dimension.
Here, we will stress the dependency by writing 
$$\Ups_p = \{ u\subset \{1,\dots,p\}\}$$ and
$\Mb_p$ the matrix such that for all $u,v\in\Ups_p$
\[[\Mb_p]_{u,v} =  (-1)^{|u|-|v|} \delta_{v\subset v}.\]

Let us remark that the elements of $\Ups_p$ can be ordered using the lexicographic order on numbers written in binary. The correspondence writes
\begin{align*}
\emptyset=0,& \quad\{1\}=1\\
\{2\} = 10, & \quad\{1,2\} =11\\
\{3\} = 100, &\quad \{1,3\}=101, \quad \{2,3\} = 110, \quad\{1,2,3\}=111\cdots
\end{align*}
Furthermore 
\[ \Ups_{p+1} = \Ups_p \cup \left\{ u\cup\{p+1\}, u\in\Ups_p\right\}\]
This recursion allows to write a recursion on matrices $\Mb_p$ when rows and columns are ordered as above:
\begin{equation}
\label{eq:rec_M}
\Mb_{p+1}=\begin{pmatrix}
\Mb_p&0\\
-\Mb_p &\Mb_p
\end{pmatrix}
.\end{equation}

\paragraph{Invert of $\Mb$}
Upon closer examination of the matrix $\Mb$,  the general term $\Mb_{u,v}$ corresponds to the Möbius function for subsets ordered by inclusion\footnote{This observation was kindly shared with us by P. Rochet, to whom we extend our sincere thanks.}, denoted by $\mu_{\Ups}(v,u) = (-1)^{|u|-|v|} \mathbf{1}_{v \subset u}$.
This, in turn, allows for the use of the Rota–Möbius inversion formula \cite{rot64} that states that for two functions $f$ and $g$:
$$
\forall v \in \Ups: \quad 
g(v) = \sum_{v \subset u} f(u) \Longleftrightarrow 
\forall u \in \Ups: \quad  f(u) = \sum_{v \subset u} \mu_{\Ups}(v,u) g(u).
$$
As a consequence 
$$ \forall (u,v) \in \Ups \times \Ups \qquad \Mb^{-1}_{u,v}=  \mathbf{1}_{v \subset u}.
$$
\paragraph{Eigenvalues of $\Mb_p\Mb_p^T$}
From \eqref{eq:rec_M}, we deduce that if $V_p=\Mb_p\Mb_p^T$, then 
\[V_{p+1}= \begin{pmatrix}
V_p& -V_p\\-V_p& 2V_p
\end{pmatrix}.\]
\begin{lemma}
\label{lem:vp}
For any $p\ge1$ the eigenvalues of $V_p=\Mb_p\Mb_p^T$ can be obtained by a recursion as
\[Sp(V_0)=\{1\}, \quad Sp(V_{p+1})=\left\{\frac{3+\sqrt{5}}{2}\lambda, \frac{3-\sqrt{5}}{2} \lambda, \lambda\in Sp(V_p) \right\}.
\]
In particular the largest eigenvalue of $V_p$ equals $\left(\frac{3+\sqrt{5}}{2}\right)^p$ and its smallest is $\left(\frac{3-\sqrt{5}}{2}\right)^p=\left(\frac{3+\sqrt{5}}{2}\right)^{-p}$.
\end{lemma}
\begin{proof}
Let $\lambda$ be an eigenvalue of $V_p$ and $u$ an eigenvector, then for $\mathfrak{a}\in\R$
\[V_{p+1}\begin{pmatrix}
u\\\mathfrak{a} u
\end{pmatrix}=\begin{pmatrix}
V_p& -V_p\\-V_p& 2V_p
\end{pmatrix}\begin{pmatrix}
u\\\mathfrak{a} u
\end{pmatrix} = \begin{pmatrix}
\lambda (1-\mathfrak{a} ) u \\\lambda(-1+2\mathfrak{a})u
\end{pmatrix}= \begin{pmatrix}
\lambda (1-\mathfrak{a} ) u \\\lambda\frac{(-1+2\mathfrak{a})}{\mathfrak{a}} au
\end{pmatrix}.\]
As a consequence, $\begin{pmatrix}
u\\ \mathfrak{a} u
\end{pmatrix}$ is an eigenvector of $V_{p+1}$ if and only if \[1-\mathfrak{a}= \frac{(-1+2\mathfrak{a})}{\mathfrak{a}} \iff \mathfrak{a}^2+\mathfrak{a}-1=0.\]
This amounts to choose $\mathfrak{a}=\frac{-1 \pm \sqrt{5}}{2}.$
We deduce from this reasoning that if $\lambda$ is an eigenvalue of $V_p$ then $\frac{3\pm \sqrt{5}}{2}\lambda$ are eigenvalues of $V_{p+1}$.
By recursion we obtain all the eigenvalues of $V_{p+1}$.
\end{proof}

\begin{prop}
\label{prop:hessienne}
For any distribution $a$ and any $s\in \Delta_q$ the Hessian Matrix $\nabla^2\Phi^a(s)$ verifies:
\[\nabla^2 \Phi^a(s)_{i,j} = \Var(Y)\E[\Mbm _{U,i}\Mbm _{U,j}]=\Var(Y)\left(\sum_{u\subset\{1,\cdots, p\}}a_u \Mbm _{u,i}\Mbm _{u,j} \right ), \quad \forall i,j\in\Ups.\]
Moreover, the lowest eigenvalue of $\nabla^2 \Phi^a(s)$ verifies:
$$
\inf \{ \lambda \, : \lambda \in Sp(\nabla^2 \Phi^a(s))\} \ge 
\Var(Y) \left(\frac{3-\sqrt{5}}{2}\right)^p \min_{u\in\Ups} a_u.
$$
\end{prop}
\begin{proof}
The formula for the Hessian matrix derives from a simple computation. 
Concerning the smallest eigenvalue with a lower bounded discrete probability distribution $a$, we can apply Lemma \ref{lem:vp} and obtain that the lowest eigenvalue is $ \Var(Y)  \times \min_{u\in\Ups} a_u \times \left(\frac{3-\sqrt{5}}{2}\right)^p$. 
\end{proof}

\noindent Let us remark that in the case of a uniform distribution $a$, the computation can be carried out explicitly:
\[\nabla^2 \Phi^{(unif)}(s) = \frac{\Var(Y)}{2^p}(\Mb\Mb^T)^{-1}. \]

\section*{Acknowledgments}
 
The authors thank Paul Rochet for recognizing the Möbius function as well as Sébastien Da Veiga for giving them access to the results of the numerical experiments performed in \cite{da_veiga_efficient_2024} and running his R code on the test function of Section 4.2.

\def\cprime{$'$}

\end{document}